\documentclass[12pt]{amsart}
\usepackage{graphicx}
\usepackage{amsmath}
\usepackage{amssymb}
\usepackage{amsfonts}
\usepackage{mathtools}
\usepackage{amsthm}  
\usepackage{dsfont}
\usepackage[colorlinks]{hyperref}
\usepackage{mathrsfs}
\usepackage{anysize}
\marginsize{1.5cm}{1.5cm}{2cm}{2cm}

\usepackage{tikz}

\usepackage{bm}

\usepackage{enumitem}

\newcommand{\R}{\mathbb{R}}
\newcommand{\C}{\mathbb{C}}

\newcommand{\N}{\mathbb{N}}

\newcommand*{\longeq}{\ratio\Longleftrightarrow}

\newcommand{\supp}[1]{\operatorname{supp}\left(#1\right) }

\theoremstyle{plain}
\newtheorem{theorem}{Theorem}[section]

\newtheorem*{theorem*}{Theorem}
\newtheorem{proposition}[theorem]{Proposition}
\newtheorem{corollary}[theorem]{Corollary}
\newtheorem{lemma}[theorem]{Lemma}

\theoremstyle{definition} 
\newtheorem{definition}[theorem]{Definition} 
\newtheorem{remark}[theorem]{Remark} 

\numberwithin{equation}{section}

\title[\texorpdfstring{B.V.P.'s with $A_{\infty}$-measures on the boundary}{}]{\texorpdfstring{Boundary Value Problems in graph Lipschitz domains in the plane with $A_{\infty}$}{}-measures on the boundary}

\author[F. Ballesta-Yag\"ue]{Fernando Ballesta-Yag\"ue}
\address[F. Ballesta-Yag\"ue]{Departamento de An\'alisis Matem\'atico y Matem\'atica Aplicada, Facultad de Ciencias Matem\'aticas, Universidad Complutense de Madrid \hfill\break\indent Pl. de las Ciencias 3, 28040 Madrid, Spain}
\email{ferballe@ucm.es}

\author[M. J. Carro]{Mar\'ia J. Carro}
\address[M. J. Carro]{Departamento de An\'alisis Matem\'atico y Matem\'atica Aplicada, Facultad de Ciencias Matem\'aticas, Universidad Complutense de Madrid \hfill\break\indent Pl. de las Ciencias 3, 28040 Madrid, Spain}
\email{mjcarro@ucm.es}

\allowdisplaybreaks
\begin{document}

 

\thanks{Partially supported by grants PID2020-113048GB-I00 funded by MCIN/AEI/10.13039/501100011033,  CEX2019-000904-S funded by MCIN/AEI/ 10.13039/501100011033 and Grupo UCM-970966 (Spain). The first named author also benefited from an FPU Grant FPU21/06111 from Ministerio de Universidades (Spain).}

\begin{abstract}
    We prove several results for the Dirichlet, Neumann and Regularity problems for the Laplace equation in graph Lipschitz domains in the plane, considering $A_{\infty}$-measures on the boundary. More specifically, we study the $L^{p,1}$-solvability for the Dirichlet problem, complementing results of \cite{kenigWeighted} and \cite{CarroOrtiz}. Then, we study $L^p$-solvability of the Neumann problem, obtaining a range of solvability which is empty in some cases, a clear difference with the arc-length case. When it is not empty, it is an interval, and we consider solvability at its endpoints, establishing conditions for Lorentz space solvability when $p>1$ and atomic Hardy space solvability when $p=1$. Solving the Lorentz endpoint leads us to a two-weight Sawyer-type inequality, for which we give a sufficient condition. Finally, we show how to adapt to the Regularity problem the results for the Neumann problem.
\end{abstract}

\maketitle

\section{Introduction} \label{sec:section1_intro}

From the late 70's, beginning with the Dirichlet problem for the Laplace equation \cite{DahlbergHarmonicMeasure}, there has been an intensive attempt to study boundary value problems in domains with low regularity.  Let us just mention for now the recent  work \cite{AHMMT2020}, where the authors establish the minimal assumptions on the domain for the Dirichlet problem for the Laplacian to be well-posed in $L^p$ for some $p<\infty$, and the works of Mourgoglou and Tolsa \cite{MTregularity, MTneumann}, where they solve the $L^p$-Neumann and Regularity problems in dimensions $3$ and higher in a more general context than that of Lipschitz domains. 

In his PhD thesis \cite{kenigWeighted}, Kenig gave an application of weighted theory to this kind of problems. More specifically, he used the theory of Muckenhoupt's $A_p$-weights \cite{M1972}, together with results about weighted Hardy spaces, to study the Dirichlet problem in graph Lipschitz domains with $A_{\infty}$-measures on the boundary, obtaining the following theorem. For a detailed exposition of the notation involved, see Section \ref{sec:notation}.

\begin{theorem}[Theorem 4.4 in \cite{kenigWeighted}]\label{thm:thm4.4_Kenig}
    Let $\Omega$ be a graph Lipschitz domain and let $\Phi\colon \R_{+}^{2}\to \Omega$ be the conformal mapping specified in Section \ref{sec:notation}. Let $\nu\in A_{\infty}(\Lambda)$ and let 
 \begin{equation}\label{pfi}
    p_{\Phi}(\nu)
    \coloneqq 
    \inf\{q\in (1,\infty): \Phi(\nu)\in A_{q}(\R)\}.
\end{equation}
    where $\Phi(\nu)$ denotes $(\nu\circ \Phi)|\Phi'|$.     
    If $p\in (p_{\Phi}(\nu),\infty)$, given any $f \in L^p(\Lambda,  \nu)$, there exists $u$ harmonic in $\Omega$ so that $u$ converges to $f$ non-tangentially and  $\left\|\mathcal{M}_{\alpha_0} u\right\|_{L^p(\Lambda,  \nu)} \lesssim \|f\|_{L^p(\Lambda, \nu)}$. Moreover, if $p\in (1,p_{\Phi}(\nu)]$, there exists $f \in  L^p(\Lambda,  \nu)$ such that there exists no $u$ harmonic in $\Omega$ so that $u$ converges to $f$ non-tangentially and $\mathcal{M}_{\alpha_0} u \in L^p(\Lambda,  \nu)$.
\end{theorem}
The idea of the proof is to consider the problem in $\R_{+}^{2}$ with boundary datum $f\circ\Phi$. In $\R_{+}^{2}$, a explicit solution is the Poisson integral of $f\circ\Phi$, and the non-tangential maximal operator of this Poisson integral can be bounded by $\mathcal{M}_{hl}(f\circ\Phi)$. So if $\Phi(\nu)\in A_p(\R)$,  this operator is bounded in $L^p(\Phi(\nu))$. Therefore, one can solve the Dirichlet problem in $\R_{+}^{2}$ with measure $\Phi(\nu)$ on $\R\equiv\partial\R_{+}^{2}$, and transfer the estimates back to $\Omega$ via $\Phi$. Conversely, if the problem is solvable in $L^p(\Lambda,\nu)$, then it can be proved that the Hilbert transform is bounded in $L^p(\Phi(\nu))$, and hence $\Phi(\nu)\in A_p(\R)$ (\cite{HMW}).

This technique reached its peak in \cite{jerisonKenig}, where the authors stated the $L^p$-solvability of the Dirichlet, Neumann and Regularity problems in chord-arc domains in the plane with arc-length measure on the boundary. However, it was superseded later by other methods, such as layer potentials (\cite{VerchotaLayer, DahlbergKenig}) or harmonic measure (\cite{JK1982,FJK1984}), that also worked in higher dimensions and for more general differential operators.

Nevertheless, the conformal mapping technique has recently been recovered in \cite{CarroOrtiz}, where they prove the following endpoint result for the Dirichlet problem for the Laplacian in the case of arc-length measure $ds$ (i.e. $\nu=1$). 

\begin{theorem}[Theorem 2.4 in \cite{CarroOrtiz}]\label{thm:thm2.4_CarroOrtiz}
    In the setting of Theorem \ref{thm:thm4.4_Kenig}, if $|\Phi'|\in A_{p_{\Phi}(ds)}^{\mathcal{R}}(\R)$, then the Dirichlet problem for the Laplacian is solvable in $X=L^{p_{\Phi}(ds),1}(\Lambda,ds)$ taking $Y=L^{p_{\Phi}(ds),\infty}(\Lambda,ds)$.
\end{theorem}

 Followed by this work,  the conformal mapping technique has been used in \cite{CNO}, \cite{zarembaCarro} and  \cite{transmissionCarro}  to  study $L^p$-solvability results and obtain new endpoint results for some boundary value problems (B.V.P.'s from now on) for the Laplace equation. This approach provides a way to understand B.V.P.'s  via weighted theory and, conversely, it motivates the study of several problems involving weights.  However, the results obtained so far, in the case of the Dirichlet and Neumann problems, only treat the case of arc-length measure on the boundary, instead of the more general $A_{\infty}$-measures originally considered by Kenig in Theorem \ref{thm:thm4.4_Kenig}.

The main goal of this paper is to address this question for the $L^{p,1}$-endpoint of the Dirichlet problem, the $L^p$-solvability for the Neumann problem, and its corresponding $L^{p,1}$ and $H_{at}^{1}$-endpoints. As a consequence, we obtain the $L^p$ and $L^{p,1}$-solvability of the the Regularity problem.
 
 We should mention here that, for the $L^p(\Lambda,\nu)$-Neumann and Regularity problems, we obtain a range of solvability which is empty in some cases, a clear difference with the arc-length case. When it is not empty, it is an interval, and hence it makes sense to consider solvability at two endpoints with boundary data in Lorentz spaces when $p>1$, and in $H_{at}^{1}$ when $p=1$. Moreover, solving the Lorentz endpoints leads us to a new two-weight Sawyer-type inequality.

\subsection{Main results}

The first result is a quite direct generalization of Theorem \ref{thm:thm2.4_CarroOrtiz}. We will always work with the Laplacian, so we omit it from the statements. See Section \ref{sec:notation} for the definitions involved.

\begin{theorem}\label{thm:dirichletEndpointPesos}
    Let $\nu\in A_{\infty}(\Lambda)$. If $\Phi(\nu)\in A_{p_{\Phi}(\nu)}^{\mathcal{R}}(\R)$, then the Dirichlet problem is solvable in $X=L^{p_{\Phi}(\nu),1}(\Lambda,d\nu)$ taking $Y=L^{p_{\Phi}(\nu),\infty}(\Lambda,d\nu)$.
\end{theorem}

The $L^p$-solvability for the Neumann problem in terms of weights when $p\in (1,\infty)$ is given by the following theorem. For the notion of unique solvability for the Neumann problem, see Definition \ref{def:definitionSolvability}.
\begin{theorem}\label{thm:Lp_Solvability}
    Let $\nu\in A_{\infty}(\Lambda)$. If $|\Phi'|^{-p}\Phi(\nu)\in A_{p}(\R)$, then the Neumann problem is uniquely solvable in $X=L^p(\Lambda,d\nu)$ taking $Y=L^p(\Lambda,d\nu)$.
\end{theorem}
In view of this theorem, given $\nu\in A_{\infty}(\Lambda)$, we  define the range of $L^p$-solvabilty
\[
R(\nu)
\coloneqq 
\{ 
p\in (1,\infty): |\Phi'|^{-p}\Phi(\nu)\in A_p(\R)
\}.
\]
When the weight is $\nu=1$, we recover Theorem 1.4 in \cite{CNO}. Indeed, in this case, 
\[
R(ds)
=
\{p\in (1,\infty): |\Phi'|^{1-p}\in A_p(\R)\}
=
\{p\in (1,\infty): |\Phi'|\in A_{p'}(\R)\}, 
\]
and hence, by the nesting properties of the $A_p$-classes,  $R(ds)$ is an interval of the form $(1,p_{\Phi}')$, where $p_\Phi=p_\Phi(ds)$ is as in \eqref{pfi}. We observe that, since $|\Phi'|\in A_2(\R)$ \cite{calderonCauchyIntegral, kenigWeighted}, $p_{\Phi}<2$ and therefore $p_{\Phi}'>2$. This recovers the classical result that the Neumann problem in a Lipschitz domain is solvable in $(1,2+\varepsilon)$ for some $\varepsilon>0$ depending on the domain \cite{DahlbergKenig}, giving an explicit formula for such $\varepsilon$. 

For a general weight $\nu\in A_{\infty}(\Lambda)$, the range of solvability $R(\nu)$ may be empty, as we will see in examples in Section \ref{section:examplesDirichlet}, and, when it is not empty, it is an open interval (see Theorem \ref{lemma:lemma2RangeSolvability}). Therefore, we can consider
\[
p_{+}\coloneqq \sup R(\nu)\quad \text{ and }\quad p_{-}\coloneqq \inf R(\nu).
\]
One may wonder whether the Neumann problem is solvable at $p_{\pm}$ for some space smaller than $L^p(\Lambda,\nu)$. For the case $p_-=1$,  we obtain the following Hardy space result.

\begin{theorem}\label{thm:H1solvability}
    Let $\nu\in A_{\infty}(\Lambda)$. If $(\nu\circ\Phi)\in A_1$, then the Neumann problem is uniquely solvable in $X=H_{at}^1(\Lambda,\nu)$ taking $Y=L^1(\Lambda,\nu)$.
\end{theorem}

Observe that $\nu\circ\Phi\in A_1$ is the case $p=1$ of the condition imposed in Theorem \ref{thm:Lp_Solvability}. Also notice that, if $d\nu=ds$ is arc-length measure,  $\nu= 1\in A_1$, so we have the following direct consequence.

\begin{corollary}
The Neumann problem is uniquely solvable in $X=H_{at}^{1}(\Lambda,ds)$ with $Y=L^1(\Lambda,ds)$ for any graph Lipschitz domain in the plane $\Omega$. 
\end{corollary}

Again, this recovers the classical $H_{at}^{1}$-result in \cite{DahlbergKenig}, but now using the conformal mapping technique. In fact, Dahlberg and Kenig combine this endpoint with the $L^2$-solvability to prove $L^p$-solvability for $p\in (1,2)$, using interpolation. However, we follow the approach in \cite{CNO}, and therefore we do not need $H_{at}^1$-solvability to prove the $L^p$-result.

For the case $p_{\pm}\in (1,\infty)$, we have the following Lorentz-space result, whose hypothesis will be explained afterwards:

\begin{theorem}[Solvability in $L^{p_{\pm},1}(\Lambda,\nu)$]\label{thm:Lp1_Solvability}
    Let $\nu\in A_{\infty}(\Lambda)$ be such that $R(\nu)=(p_{-},p_{+})\neq\varnothing$.
    \begin{itemize}
        \item If $p_{+}<\infty$ and $(|\Phi'|,\Phi(\nu))\in \mathcal{S}_{p_{+}}^{\mathcal{R}}$, then the Neumann problem is solvable in $X=L^{p_{+},1}(\Lambda,\nu)$ with $Y=L^{p_{+},\infty}(\Lambda,\nu)$.
        \item If $p_{-}>1$, $(|\Phi'|,\Phi(\nu))\in \mathcal{S}_{p_{-}}^{\mathcal{R}}$ and \eqref{eq:conditionForWellDefined} holds, then the Neumann problem is solvable in $X=L^{p_{-},1}(\Lambda,d\nu)$ with $Y=L^{p_{-},\infty}(\Lambda, \nu)$. 
    \end{itemize}
\end{theorem}

Condition $\mathcal{S}_{p}^{\mathcal{R}}$ on the pair of weights $(|\Phi'|,\Phi(\nu))$ means that the operator 
\[
f\mapsto \frac{\mathcal{A}_{\mathcal{S}}(f|\Phi'|)}{|\Phi'|}
\]
is bounded from $L^{p,1}(\Phi(\nu))\to L^{p,\infty}(\Phi(\nu))$, where $\mathcal{A}_{\mathcal{S}}$ denotes a sparse operator. In view of Remark \ref{remark:SpFuerte}, this condition is an endpoint case of the one given in Theorem \ref{thm:Lp_Solvability}.

This kind of weighted estimates are usually known as Sawyer-type inequalities or mixed weak-type inequalities \cite{SawyerCUMP, LiOmbrosiPerez, perezRoure}, and they have recently appeared in a similar form under the name of multiplier inequalities \cite{CruzUribeSweeting}. Here, they are directly motivated by a boundary value problem with a measure on the boundary, showing that applying the conformal mapping strategy to B.V.P.'s gives rise to interesting problems in weighted theory.

Finally, the $L^p$ and $L^{p,1}$ results for the Neumann problem can be adapted to the Regularity problem.

\begin{theorem}\label{thm:regularitySolvability}
    Let $\nu\in A_{\infty}(\Lambda)$. The Regularity problem is:
    \begin{itemize}
        \item Uniquely solvable in $X=L^p(\Lambda,\nu)$ taking $Y=L^p(\Lambda,\nu)$ if $|\Phi'|^{-p}\Phi(\nu)\in A_p$.
        \item Solvable in $X=L^{p_{+},1}(\Lambda,\nu)$ with $Y=L^{p_{+},\infty}(\Lambda,\nu)$ if $R(\nu)=(p_{-},p_{+})\neq\varnothing$ with $p_{+}<\infty$, $(|\Phi'|,\Phi(\nu))\in \mathcal{S}_{p_{+}}^{\mathcal{R}}$. The same holds for $p_{-}$ with the same hypothesis if $p_{-}>1$ and, besides, \eqref{eq:conditionForWellDefined} holds.
    \end{itemize}
\end{theorem}

\section{Notation and preliminaries}\label{sec:notation}

\subsection{Graph Lipschitz Domains}

Let $\Lambda$ be a curve in the complex plane given parametrically by $\eta(x)=x+i\gamma(x)$ for $x \in \mathbb{R}$, where $\gamma$ is a real-valued Lipschitz function with constant $L$, and consider the Lipschitz domain
\begin{equation}\label{eq:omega}
\Omega
=
\left\{
z_1+i z_2 \in \mathbb{C}: z_2>\gamma(z_1)
\right\}.    
\end{equation}
Note that $\partial \Omega=\Lambda$.

By Rademacher's theorem, a Lipschitz function is differentiable almost everywhere. Therefore, it makes sense to talk about the tangent vector $\eta'(x)$ to $\Lambda$ at a point $\xi=\eta(x)$ $dx$-a.e., where $dx$ denotes Lebesgue measure. Hence, we can define the outer normal vector $\mathbf{n}(\xi)$ at almost every $\xi\in \Lambda$ with respect to arc-length measure. We define arc-length measure on $\Lambda$ as
\[
ds(E)
\coloneqq 
\int_{\eta^{-1}(E)}|\eta'(t)|\, dt
=
\int_{\eta^{-1}(E)}\sqrt{1+|\gamma'(t)|^2}\, dt,\quad E\subset\Lambda.
\]

To approach the boundary from inside the graph Lipschitz domain, the following non-tangential cones are introduced:

\begin{definition}
Given $0<\alpha<\arctan (\frac{1}{L})$, we define the non-tangential cone $\Gamma_{\alpha}(\xi)$ as
\[
\Gamma_\alpha(\xi)=\left\{z_1+i z_2 \in \mathbb{C}: z_2>\operatorname{Im}(\xi) \text { and }\left|\operatorname{Re}(\xi)-z_1\right|<\tan (\alpha)\left|z_2-\operatorname{Im}(\xi)\right|\right\}.
\]     
\end{definition}

\begin{definition}
A function $v$ defined on $\Omega$ converges non-tangentially to a function $g$ defined on $\Lambda$ if there exists $0<\alpha<\arctan(\frac{1}{L})$ such that 
\[
\lim _{\substack{z \in \Gamma_\alpha(\xi)\\ z \rightarrow \xi}} v(z)
=
g(\xi)
\quad ds\text{-a.e. } \xi \in \Lambda.
\]    
\end{definition}

In the sequel, we use $z \triangleright \xi$ to denote that $z \rightarrow \xi$ with $z \in \Gamma_\alpha(\xi)$ for some $0<\alpha<\arctan (\frac{1}{L})$.

\begin{definition}
Given $0<\alpha<\arctan (\frac{1}{L})$, define the non-tangential maximal operator $\mathcal{M}_\alpha$ as
\[
\mathcal{M}_\alpha(F)(\xi)=\sup _{z \in \Gamma_\alpha(\xi)}|F(z)|, \quad \xi \in \Lambda,
\]
where $F$ is a complex-valued function defined in $\Omega$.
\end{definition}

\subsection{Conformal Mapping}

Since the region $\Omega$ is simply connected, it is conformally equivalent to the upper half-plane
\[
\R_{+}^{2}
\coloneqq 
\{
x+iy\in\C: x\in\R, y>0
\}.
\]
Let $z_0=i x_0, x_0>\eta(0)$, and $\Phi: \R_{+}^{2} \rightarrow \Omega$ be the conformal mapping such that $\Phi(\infty)=\infty$, and $\Phi(i)=z_0$. Let $\Phi^{-1}: \Omega \rightarrow \R_{+}^{2}$ be its inverse.

\begin{theorem}[Theorem 1.1 in \cite{kenigWeighted}]\label{thm:thm1.1_Kenig}
With the above notation:
\begin{enumerate}
    \item $\Phi$ extends to $\overline{\R_{+}^{2}}$ as a homeomorphism onto $\overline{\Omega}$.
    \item $\Phi'$ has a non-tangential limit $dx$-a.e. on $\R$, and this limit is different from 0 $dx$-a.e. on $\R$.
    \item $\Phi(x)$, with $x\in\R=\partial\R_{+}^{2}$, is absolutely continuous when restricted to any finite interval, and hence $\Phi'(x)$ exists $dx$-a.e. and is locally integrable. Moreover, $\Phi'(x)=\lim_{z\triangleright x} \Phi'(z)$ $dx$-a.e.
    \item Sets of $dx$-measure 0 on $\R$ correspond by $\Phi\colon \R\to\Lambda$ to sets of $ds$-measure 0 on $\Lambda$, and vice versa.
    \item At every point where $\Phi'(x)$ exists and is different from $0$, it is a vector tangent to the curve $\Lambda$ at the point $\Phi(x)$, and hence $|\arg \Phi'(x)|\leq \arctan{L}$ $dx$-a.e.
\end{enumerate}    
\end{theorem}

We abuse notation and denote by $\Phi$ the conformal mapping on $\R_{+}^{2}$, the extension to $\overline{\R_{+}^{2}}$ and the restriction to $\R\equiv \partial\R_{+}^{2}$. We will use this conformal mapping and its inverse $\Phi^{-1}$ to transfer the boundary value problems from $\Omega$ to $\R_{+}^{2}$ and viceversa. 

Conformal mappings from $\R_{+}^{2}$ onto graph Lipschitz domains $\Omega$ preserve non-tangential cones in the following sense: for any $\alpha\in (0,\arctan{\frac{1}{L}})$, there exists $\beta,\gamma>0$ depending only on $\alpha$ such that for any $x\in\R\equiv \partial\R_{+}^{2}$, 
\begin{equation} \label{eq:thm1.1_JerisonKenig}
    \Gamma_{\beta}(\Phi(x),\Omega)
    \subset
    \Phi(\Gamma_{\alpha}(x,\R_{+}^{2}))
    \subset 
    \Gamma_{\gamma}(\Phi(x),\Omega).
\end{equation}
This result follows from  Proposition 1.1 in \cite{jerisonKenig}, because Lipschitz domains satisfy the \textit{three point condition} and conformal mappings are $1$-quasiconformal.

\subsection{Dirichlet, Neumann and Regularity problems}

\begin{definition}
For a locally integrable function $g$ on $\Lambda$, a function $v$ defined in $\Omega$ is a solution of the Dirichlet problem
in $\Omega$ with datum $g$, and we write
\begin{equation}\label{eq:Dirichlet}
    \Delta v=0\text{ in }\Omega\quad \text{ and }\quad v=g\text{ on }\partial\Omega,
\end{equation}
if $v$ is harmonic in $\Omega$ and $v$ converges non-tangentially to $g$.
\end{definition}

\begin{definition}
For a locally integrable function $g$ on $\Lambda$, a function $v$ defined in $\Omega$ is a solution of the Neumann problem
in $\Omega$ with datum $g$, and we write
\begin{equation}\label{eq:Neumann}
    \Delta v=0\text{ in }\Omega\quad \text{ and }\quad \nabla v\cdot \mathbf{n}=g\text{ on }\partial\Omega,
\end{equation}
 if $v$ is harmonic in $\Omega$ and 
\[
\lim _{z \triangleright \xi} \nabla v(z)\cdot \mathbf{n}(\xi)=g(\xi)\quad \text{ for } ds\text{-a.e. } \xi \in \Lambda.
\]
\end{definition}

\begin{definition}\label{def:definitionSolvability}
If $X$ is a Banach space of measurable functions defined on $\Lambda$, we say that the Dirichlet problem in $\Omega$ is solvable in $X$ if there exist a Banach space $Y$ of measurable functions defined on $\Lambda$ and $0<\alpha<\arctan (\frac{1}{L})$ such that for every $g \in X$ there exists a solution $v_g$ of the Dirichlet problem in $\Omega$ with datum $g$, and satisfies the estimate
\[
\left\|\mathcal{M}_\alpha (v_g)\right\|_Y
\lesssim
\|g\|_X.
\]
The definition for the Neumann problem is the analogous one, but the estimate is instead
\[
\left\|\mathcal{M}_\alpha(\nabla v_g)\right\|_Y \lesssim\|g\|_X.
\]
We say that the Neumann problem is uniquely solvable if the solution is unique up to additive constants.
\end{definition}

The notation $A\lesssim B$ means that there exists a constant $C>0$ such that $A\leq C\cdot B$. If we want to make explicit the dependence of $C$ on some parameter $\alpha$, we may write $A\lesssim_{\alpha} B$. We write $A\simeq B$ to denote $A\lesssim B$ and $B\lesssim A$.

For the Regularity problem, also known as Dirichlet Regularity problem, we need to be a little more careful with the definition in order to be rigurous in our statements. We introduce the following space, which is the biggest class of functions on $\Lambda$ for which we can take the Poisson integral after composing with $\Phi$.

\begin{definition}\label{def:spaceIntegrableFunctions}
    Given $\Lambda$ and $\Phi$ as above, we define
    \[
    \widetilde{L}^{1}(\Lambda)
    \coloneqq 
    \left\{
    f\in L_{loc}^{1}(\Lambda): \int_{\R}\frac{|(f\circ\Phi)(x)|}{1+|x|^2}\, dx<\infty
    \right\}.
    \]
\end{definition}

\begin{definition}\label{def:definitionSolvabilityRegularity}
If $X$ is a Banach space of measurable functions defined on $\Lambda$, we say that the Regularity problem in $\Omega$ is solvable in $X$ if there exists a Banach space $Y$ of measurable functions defined on $\Lambda$ and $0<\alpha<\arctan (\frac{1}{L})$ such that, for every $g \in \widetilde{L}^1(\Lambda)$ with weak derivative $g'\in X$, there exists a solution $v_g$ of the Dirichlet problem in $\Omega$ with datum $g$, and satisfies the estimate
\begin{equation}\label{eq:regularityEstimate}
\left\|\mathcal{M}_\alpha (\nabla v_g)\right\|_Y 
\lesssim
\|g'\|_X.    
\end{equation}
\end{definition}

The definition of weak derivative of a function defined on $\Lambda$ is clarified in Section \ref{section:regularity}.

\subsection{Function spaces and weights on \texorpdfstring{$\R$}{}}

A weight $w$ is a function defined on $\mathbb{R}$ that is locally integrable and positive a.e. Given $p\in (1,\infty)$, we denote by $L^p(\mathbb{R}, w)$ the space of measurable functions $f: \mathbb{R} \rightarrow \mathbb{C}$ such that
\[
\|f\|_{L^p(\mathbb{R}, w)}
=
\left(\int_{\mathbb{R}}|f(x)|^p w(x) d x\right)^{\frac{1}{p}}
<
\infty.
\]
The main class of weights is Muckenhoupt's class $A_p$ \cite{M1972}, defined for $p>1$ as
\[
w\in A_p(\R)
\longeq
\mathcal{M}_{hl}\colon L^p(w)\to L^p(w),
\]
and for $p=1$ as
\[
w\in A_1(\R)
\longeq
\mathcal{M}_{hl}\colon L^1(w)\to L^{1,\infty}(w).
\]
Here, $\mathcal{M}_{hl}$ denotes the Hardy-Littlewood maximal operator, defined for $f\in L_{\rm{loc}}^{1}(\R)$ as
\[
\mathcal{M}_{h l} f(x)
=
\sup_{x\in I}\frac{1}{|I|}\int_I|f(x)|\, dx,
\]
where the supremum is taken over all intervals $I \subset \mathbb{R}$ that contain $x$ and, for a measurable set $A \subset \mathbb{R}$, $|A|$ denotes the Lebesgue measure of $A$.

The class $A_{\infty}(\R)$ is defined as
\[
A_{\infty}(\R)
\coloneqq
\bigcup_{p\geq 1}A_{p}(\R).
\]
We will sometimes abbreviate $A_{p}(\R)$ by $A_p$, for any $p\in [1,\infty]$.

Some standard properties of $A_p$ weights (see page 218 of \cite{steinHarmonic} and page 546 of \cite{grafakosClassical}) are the following.

\begin{lemma}\label{lemma:improvingA1}    Let $p\in [1,\infty)$ and $u\in A_p$. 
\begin{enumerate}
\item  Then there exists $r>1$ small enough such that $u^r\in A_p$.
\item There exist $A_1$ weights $w_1$ and $w_2$ such that
\[
w=w_1 w_2^{1-p} .
\]
\end{enumerate}

Moreover, if $w_0,w_1\in A_p$ and $\alpha_0,\alpha_1>0$ satisfy $\alpha_0+\alpha_1\leq 1$, then $w_0^{\alpha_0}w_1^{\alpha_1}\in A_p$.

\end{lemma}

To define Poisson and Neumann integrals for functions in weighted spaces, we need the following estimate. 

\begin{lemma}[\cite{CNO}]\label{lemma:estimateNeumannIntegral}
    Let $p\in [1,\infty)$. If $w\in A_p$, then $\int_{\R}\frac{|f(x)|}{1+|x|}\, dx<\infty$ for every $f\in L^p(w)$.
\end{lemma}

For the endpoint results, we need to introduce Lorentz spaces. The Lorentz spaces $L^{p,1}(\R,w)$ and $L^{p,\infty}(\R,w)$ are defined as the sets of measurable functions $f\colon \R\to\C$ such that
\begin{align*}
&\|f\|_{L^{p, 1}(\mathbb{R}, w)}
=
\int_0^{\infty}\left(\lambda_f^w(y)\right)^{\frac{1}{p}} \, dy
=
\frac{1}{p} \int_0^{\infty} f_w^*(t) t^{\frac{1}{p}-1}\, dt
<
\infty,
\\&
\|f\|_{L^{p, \infty}(\mathbb{R}, w)}
=
\sup_{y>0} y\left(\lambda_f^w(y)\right)^{\frac{1}{p}}
=
\sup _{t>0} t^{\frac{1}{p}} f_w^*(t)
<
\infty,
\end{align*}
respectively. Here, $\lambda_f^w$ is the distribution function of $f$ with respect to the measure $w(A)=\int_A w(x)\, dx$, defined for any measurable $A\subset \R$; that is, 
\[
\lambda_f^w(y)
=
w(\{t \in \mathbb{R}:|f(t)|>y\}).
\]
And $f_w^*$ is the decreasing rearrangement of $f$ with respect to $w$, given by 
\[
f_w^*(t)
=
\inf \left\{y>0: \lambda_f^w(y) \leq t\right\}.
\]

The class of restricted $A_p$-weights \cite{KermanTorchinsky, ChungHuntKurtz}, denoted by $A_p^{\mathcal{R}}(\R)$ or simply  $A_p^{\mathcal{R}}$, is defined as
\[
w\in A_p^{\mathcal{R}}(\R)
\longeq
\mathcal{M}_{hl}\colon L^{p,1}(w)\to L^{p,\infty}(w).
\]
It is known \cite{CarroOrtiz} that 
\[
A_p(\R)
\subsetneq 
A_p^{\mathcal{R}}(\R)
\subsetneq
\bigcup_{\varepsilon>0}A_{p+\varepsilon}(\R).
\]

These classes of weights can be generalized to consider maximal operators with respect to measures different from Lebesgue measure.

Let $u$ and $w$ be weights and let $p\in (1,\infty)$. Then, $w\in A_p(u)$ if and only if
\[
\mathcal{M}_{u}
\colon 
L^p(wu)\to L^p(wu),
\]
and $w\in A_p^{\mathcal{R}}(u)$ if and only if
\[
\mathcal{M}_{u}
\colon 
L^{p,1}(wu)\to L^p(wu).
\]
Here, 
\[
\mathcal{M}_{u}f(x)
\coloneqq 
\sup_{Q\ni x}\frac{1}{u(Q)}\left(\int_{Q}fu\, dx\right), \quad x\in\R^n.
\]

These classes can be characterized in terms of cubes, as usual. Indeed, $w\in A_p(u)$ if and only if
\[
\sup_{Q\subset\R^n}\frac{1}{u(Q)}\|\mathds{1}_{Q}\|_{L^p(wu)}\|w^{-1}\mathds{1}_{Q}\|_{L^{p'}(wu)}
<
\infty,
\]
and $w\in A_p^{\mathcal{R}}(u)$ if and only if
\[
\sup_{Q\subset\R^n}\frac{1}{u(Q)}\|\mathds{1}_{Q}\|_{L^{p,1}(wu)}\|w^{-1}\mathds{1}_{Q}\|_{L^{p',\infty}(wu)}
<
\infty.
\]
This last condition can be rewritten using Kolmogorov's inequalities, obtaining that 
\[
w\in A_p^{\mathcal{R}}(u)
\iff 
\sup_{Q\subset\R^n}\sup_{E\subset Q}\frac{u(E)}{u(Q)}\left(\frac{uw(Q)}{uw(E)}\right)^{\frac{1}{p}}
<
\infty.
\]

\subsubsection{Sparse operators}

A class of operators that has gained importance in recent years is that of sparse operators, because they satisfy in most cases the same weighted bounds as $\mathcal{M}_{hl}$ and many interesting operators can be bounded by a sparse one. Hence, sparse bounds for an operator imply weighted bounds for it. Gentle introductions to this topic are \cite{LernerNazarov,PereyraSparse}. 

\begin{definition}
Given $0<\eta<1$, a family $\mathcal{S}$ of cubes contained in $\mathbb{R}^n$ is said to be $\eta$-sparse if for every $Q \in \mathcal{S}$, there exists a measurable set $E_Q \subset Q$ such that $\left|E_Q\right| \geq \eta|Q|$ and the sets $\left\{E_Q\right\}_{Q \in \mathcal{S}}$ are pairwise disjoint. A family of cubes is sparse if it is $\eta$-sparse for some $0<\eta<1$.   
\end{definition}

\begin{theorem}[\cite{LernerSparse,carroDomingoSalazar}]

Given a sparse family $\mathcal{S}$, the corresponding sparse operator is defined as
\[
\mathcal{A}_{\mathcal{S}}(f)(x)
\coloneqq
\sum_{Q \in \mathcal{S}} \left(\frac{1}{|Q|} \int_Q f(y)\, dy\right) \mathds{1}_{Q}(x).
\]
Then, $\mathcal{A}_{\mathcal{S}}\colon L^p(w)\to L^p(w)$ if and only if $w\in A_p$, and $\mathcal{A}_{\mathcal{S}}\colon L^{p,1}(w)\to L^{p,\infty}(w)$ if and only if $w\in A_p^{\mathcal{R}}$.
\end{theorem}

\subsection{Weights on \texorpdfstring{$\Lambda$}{} and conformal map}

In an analogous way, one can define weights on $\Lambda$ as functions $\nu\in L_{loc}^{1}(\Lambda,ds)$  that are positive $ds$-a.e. Then $L^p(\Lambda, \nu)$ is defined as
\[
\|f\|_{L^{p}(\Lambda,\nu)}
\coloneqq 
\left(\int_{\Lambda}|f(\xi)|^p\, \nu(\xi)ds(\xi)\right)^{\frac{1}{p}}.
\]
We recall the definitions of the Muckenhoupt classes of weights in the setting of a Lipschitz curve $\Lambda$, which recover the classical characterization in $\R$ when $\Lambda=\R$.

\begin{definition}\label{def:intervalLambda}
    We define an \textit{interval} or arc in $\Lambda$ as the image of an interval $I\subset \R$ by the mapping $\eta$. That is, $J\subset\Lambda$ is an interval in $\Lambda$ if $\eta^{-1}(J)$ is an interval in $\R$. More specifically, the interval in $\Lambda$ centered at $\xi\in\Lambda$ with radius $r>0$ is defined as
    \[
    B_{\Lambda}(\xi,r)
    \coloneqq 
    \eta(B(\eta^{-1}(\xi),r)),\quad r>0, \, \xi\in\Lambda,
    \]
    where
    \[
    B(\eta^{-1}(\xi),r)
    =
    B_{\R}(\eta^{-1}(\xi),r)
    =
    \{
    x\in \R: |x-\eta^{-1}(\xi)|<r
    \}
    \]
    denotes the $\R$-interval with center $\eta^{-1}(\xi)$ and radius $r$.
\end{definition}

\begin{definition}[\cite{kenigWeighted}]\label{def:definition1.4_Kenig}
Let $p\in (1,\infty)$ and let $\nu$ be a nonnegative and locally integrable function  on $\Lambda$. We say that $\nu\in A_p(\Lambda)$, if there exists a constant $C_p>0$ such that for all intervals $J \subset \Lambda$,
\begin{equation}\label{ApLambda}
\left[\frac{1}{s(J)} \int_J \nu d s\right]
\left[\frac{1}{s(J)} \int_J \nu^{-1 /(p-1)} d s\right]^{(p-1)} 
\leq
C_p,    
\end{equation}
where $s(J)$ denotes the arc-length measure of $J$. Also, $A_{\infty}(\Lambda)\coloneqq \bigcup_{p\in (1,\infty)}A_{p}(\Lambda)$.

\end{definition}

\begin{definition}\label{prop:funcionDensidadPhiNu}
Given $\nu$  a  weight over $\Lambda$, we define the following measure on $\R$ associated to $\nu$
\[
\Phi(\nu)(E)
\coloneqq 
\nu(\Phi(E)).
\]
\end{definition}
It is immediate to check that $\frac{d\Phi(\nu)}{dx}=\left(\nu \circ\Phi\right) |\Phi'|$. We shall simply write $\Phi(\nu)$ instead of $\frac{d\Phi(\nu)}{dx}$.

\begin{lemma}[Lemma 1.16 in \cite{kenigWeighted}]\label{phiNuAInfty}
With the above notation, 
\begin{equation}\label{AinftyEquivalence}
\nu\in A_{\infty}(\Lambda)\quad
\text{ if and only if }\quad
\Phi(\nu)\in A_{\infty}(\R).    
\end{equation}
\end{lemma}

Another important fact already mentioned  in the introduction is the following.

\begin{theorem}[Theorem 1.10 in \cite{kenigWeighted}]
    $|\Phi'|\in A_2(\R)$.
\end{theorem}

\subsection{The Atomic Hardy Space}

We recall the definition of the (weighted) atomic Hardy space over $\Lambda$ with respect to a measure $\nu$ on $\Lambda$, denoted $H_{at}^{1}(\Lambda,\nu)$. This is a particular case of the general construction for spaces of homogeneous type \cite{coifmanWeiss}.

\begin{definition}\label{def:dsLambdaAtom}
    A function $a\colon\Lambda\to \C$ is an \textit{atom} over $\Lambda$ with respect to the measure $\nu$, abbreviated $(\Lambda,\nu)$-atom, if there exists an interval $B_{\Lambda}(\xi_0,r)$ on $\Lambda$,  with $\xi_0\in\Lambda$ and $r>0$, such that
    \begin{itemize}
        \item $\supp{a}\subset B_{\Lambda}(\xi_0,r)$.
        \item $\|a\|_{\infty}\leq \frac{1}{\nu(B_{\Lambda}(\xi_0,r))}$.
        \item $\int_{\Lambda}a(\xi) \nu(\xi)\, ds(\xi)=0$.
    \end{itemize}
\end{definition}

\begin{definition}
    The atomic Hardy space over $\Lambda$ with respect to $\nu$ is defined as
    \[
    H_{at}^1(\Lambda,\nu)
    \coloneqq 
    \{f\in L^1(\Lambda,\nu): f=\sum \lambda_j a_j, \text{ where } a_j \text{ is } (\Lambda,\nu)\text{-atom and }\sum|\lambda_j|<\infty\}.
    \]
    The norm associated to it is
    \[
    \|f\|_{H_{at}^1(\Lambda,\nu)}
    \coloneqq 
    \inf\left\{\sum|\lambda_j|: f=\sum \lambda_j a_j,\hspace{0.1cm} a_j\text{ atoms}\right\}.
    \]
\end{definition}

\subsection{Hardy spaces over graph Lipschitz domains}

\begin{definition}
Let $\nu$ be a measure on $\Lambda$ and let $\alpha\in (0,\arctan{\frac{1}{L}})$. Define the Hardy space 
\[
H^p(\Omega, \nu)
\coloneqq 
\left\{
h: \Omega \rightarrow \mathbb{C}: h \text { is analytic in } \Omega \text { and }\left\|\mathcal{M}_\alpha(h)\right\|_{L^p(\Lambda, \nu)}<\infty
\right\}
\]
and the weak Hardy space
\[
H^{p, \infty}(\Omega, \nu)
\coloneqq 
\left\{
h: \Omega \rightarrow \mathbb{C}: h \text { is analytic in } \Omega \text { and }\left\|\mathcal{M}_\alpha(h)\right\|_{L^{p, \infty}(\Lambda, \nu)}<\infty
\right\},
\]
and set 
\[
\|h\|_{H^p(\Omega, \nu)}
\coloneqq 
\left\|\mathcal{M}_\alpha(h)\right\|_{L^p(\Lambda, \nu)},\quad 
\|h\|_{H^{p, \infty}(\Omega, \nu)}
\coloneqq
\left\|\mathcal{M}_\alpha(h)\right\|_{L^{p, \infty}(\Lambda, \nu)}.
\]    
\end{definition}

The definitions of $H^p(\Omega, \nu)$ and $H^{p, \infty}(\Omega, \nu)$ are independent of $\alpha$ and any other chosen value gives an equivalent norm \cite{kenigWeighted, CNO}. Some basic results relating the Hardy space over a graph Lipschitz domain with the Hardy space over the upper half-plane are the following.

\begin{theorem}[\cite{kenigWeighted, CNO}]\label{thm2.8_Kenig}
    Let $p\in (0,\infty)$ and $\nu\in A_{\infty}(\Lambda)$. Then
    \[
    F\in H^p(\Omega,\nu)
    \iff 
    F\circ\Phi\in H^p(\R_{+}^{2},\Phi(\nu)),
    \]
    and
    \[
    h\in H^{p,\infty}(\Omega,\nu) \quad\iff \quad 
    h\circ\Phi\in H^{p,\infty}(\R_{+}^{2},\Phi(\nu)),
    \]
    with equivalence of norms.
\end{theorem} 

\begin{remark}\label{generalizationThm2.8_Kenig}
    The estimates of Theorem \ref{thm2.8_Kenig} are valid for functions that are not holomorphic, because the only tools used in their proofs are the behavior of non-tangential regions via $\Phi$ and the size estimate. That is, given a graph Lipschitz domain $\Omega$, there exist angles $\alpha\in (0,\arctan{\frac{1}{L}})$ and $\beta\in (0,\frac{\pi}{2})$ such that, for every function $h$ defined in $\Omega$, we have $\mathcal{M}_{\alpha}(h)\in L^p(\Lambda,\nu)$ if and only if $\mathcal{M}_{\beta}(h\circ\Phi)\in L^p(\R,\Phi(\nu))$, with equivalent norms
    \[
    \|\mathcal{M}_{\alpha}(h)\|_{L^p(\Lambda,\nu)}
    \simeq_{\alpha,\beta}
    \|\mathcal{M}_{\beta}(h\circ\Phi)\|_{L^p(\R,\Phi(\nu))}.
    \]
    The same holds for $L^{p,\infty}$.
\end{remark}

We also need the following result.

\begin{theorem}[\cite{AE_FBY}]
\label{thm:analogoLema4.1CNO}
    Let $\nu\in A_{\infty}(\Lambda)$ and $p\in [1,\infty)$. If $|\Phi'|^{-p}\Phi(\nu)\in A_\infty$, then
    \[
    F\in H^p(\R_{+}^{2},|\Phi'|^{-p}\Phi(\nu))
    \iff 
    F\cdot \frac{1}{\Phi'}\in H^p(\R_{+}^{2},\Phi(\nu)),
    \]
    with equivalent norms.
\end{theorem}

\subsection{Poisson integrals}

Finally, we introduce some notation about Poisson integrals and prove a result about logarithms of functions in $H^p(w)$ that we will need in the proof of Theorem \ref{thm:Lp1_Solvability}.

\begin{definition}
    We define the Poisson kernel at $(x,y)\in\R_{+}^{2}$ as
    \[
    P_{y}(x)
    \coloneqq 
    \frac{1}{\pi}\frac{y}{x^2+y^2}.
    \]
    Given a function $f$ with $\int_{\R}\frac{|f(x)|}{1+|x|^2}\, dx<\infty$, we define the Poisson integral of $f$ at $(x,y)\in\R_{+}^{2}$ as
    \begin{equation}\label{eq:definitionPoissonIntegral}
    (P_y\ast f)(x)
    \coloneqq 
    \frac{1}{\pi}\int_{\R}\frac{y}{(x-t)^2+y^2}f(t)\, dt.
    \end{equation}
\end{definition}

\begin{definition}
    We define the conjugate Poisson kernel at $(x,y)\in\R_{+}^{2}$ as
    \[
    Q_{y}(x)
    \coloneqq 
    \frac{1}{\pi}\frac{x}{x^2+y^2}.
    \]
    Given a function $f$ with $\int_{\R}\frac{|f(x)|}{1+|x|}\, dx<\infty$, the conjugate Poisson integral of $f$ at $(x,y)\in\R_{+}^{2}$ is
    \[
    (Q_y\ast f)(x)
    \coloneqq 
    \frac{1}{\pi}\int_{\R}\frac{x-t}{(x-t)^2+y^2}f(t)\, dt.
    \]
\end{definition}

The following control of the non-tangential maximal operator of Poisson integrals is well-known.

\begin{theorem}\label{thm:Stein_thm1Page197}
    Let $f\in L_{loc}^{1}(\R)$ with $\int_{\R}\frac{|f(x)|}{1+|x|^2}\, dx<\infty$. If $u(x,t)\coloneqq P_t\ast f(x)$ for $(x,t)\in\R_{+}^{2}$, then $u$ converges non-tangentially to $f$ a.e. in $\R$. Besides
    \[
    \mathcal{M}_{\alpha}(u)(x)
    \lesssim_{\alpha}
    \mathcal{M}_{hl}(f)(x),\quad x\in\R.
    \]
\end{theorem}

\begin{proof}
    It is proved for $L^p(\R)$-functions in Theorem 1 in page 197 of \cite{steinSingularIntegrals}. The proof also works with the above generality. 
\end{proof}

\subsection{Poisson integrals and logarithms}\label{section:logarithm}

Following \cite{garnett}, given $F\colon\mathbb{D}\to\C$, the Poisson integral of the unit disk $\mathbb{D}$ of $F$ is defined as
\[
\mathbb{P}_{\mathbb{D}}[F](w)
=
\frac{1}{2\pi}\int_{0}^{2\pi}\operatorname{Re}\left({\frac{e^{i\theta}+w}{e^{i\theta}-w}}\right)F(e^{i\theta})\, d\theta,
\quad w\in \mathbb{D},
\]
and, for convenience, we denote the Poisson integral \eqref{eq:definitionPoissonIntegral} also as 
\[
\mathbb{P}_{\R_{+}^{2}}[f](x+iy)
=
P_{y}\ast f(x), \quad x+iy\in \R_{+}^{2}.
\]

The Poisson integrals of $\mathbb{D}$ and $\R_{+}^{2}$ are related by a simple change of variables.

\begin{lemma}\label{lemma:equivalencePoissonIntegralsBis}
    Let $\Tilde{F}\colon \partial\R_{+}^{2}\equiv \R\to \R$ be a measurable function such that $\int_{\R}\frac{|\Tilde{F}(x)|}{1+|x|^2}\, dx<\infty$. Then
    \[
    \mathbb{P}_{\R_{+}^{2}}[\Tilde{F}](z)
    =
    \mathbb{P}_{\mathbb{D}}[\Tilde{F}\circ \varphi](\varphi^{-1}(z)),
    \quad z\in\R_{+}^{2}, 
    \]
    where  $\varphi\colon \mathbb{D}\to \R_{+}^{2}$
    is the conformal mapping 
    \begin{equation}\label{eq:conformalDiskToPlane}
    \varphi(w)
    =
    i\frac{1-w}{1+w}\quad w\in\mathbb{D}.
    \end{equation}
   
\end{lemma}

This change of variables is going to allow us to transfer the following result about Hardy spaces on the unit disk. 

\begin{lemma}[Theorem II.4.1 in \cite{garnett}]\label{thm:thmII.4.1_Garnett}
If $0<p \leq \infty$ and if $f(z) \in H^p(\mathbb{D})$ with $f(z)\neq 0$ for every $z\in \mathbb{D}$, then
\[
\frac{1}{2 \pi} \int_{\mathbb{T}} \log{|f(e^{i \theta})|}\, d\theta
>
-\infty.
\]
If $f\left(z_0\right) \neq 0$ for some $z_0\in\mathbb{D}$, then
\[
\log \left|f\left(z_0\right)\right| 
\leq
\mathbb{P}_{\mathbb{D}}[\log{|f|}](z_{0}).
\]
\end{lemma}

We want an analogue for the upper half-plane. It is stated for Lebesgue measure in page 62 of \cite{garnett}, and here we extend it to measures in $A_p(\R)$.

\begin{lemma}\label{thm:teoremaGarnettAdaptado}
    Let $w\in A_p(\mathbb{R})$ and $F\in H^p(\R_{+}^{2},w)$. If $F(x,t)\neq 0$ for a given $(x,t)\in\R_{+}^{2}$, then
    \[
    \log{|F(x,t)|}
    \leq 
    P_t\ast \log{|F|}(x),
    \]
    where $|F|$ on the right-hand side denotes the non-tangential limit of $|F(x,t)|$.
\end{lemma}

\begin{proof}
    Since $F\in H^p(\R_{+}^{2},w)$, by Lemma I.9 in \cite{GarciaCuerva} we have
    \[
    |F(x,t)| 
    \leq 
    \mathbb{P}_{\R_{+}^{2}}[F](x,t).
    \]
    Denote $u(x,t):= \mathbb{P}_{\R_{+}^{2}}[F](x,t)$. Since $|F|\in L^p(\R,w)$ with $w\in A_p$, $u$ is a well-defined harmonic function on $\R_{+}^{2}$, and hence $(u\circ\varphi)$ is harmonic on $\mathbb{D}$. Therefore, using the above bound and applying the mean value property for harmonic functions, we obtain
    \begin{align*}
        \|F\circ\varphi\|_{H^1(\mathbb{D})}
        \leq 
        \sup_{r\in (0,1)}\frac{1}{2\pi}\int_{\mathbb{T}}(u\circ\varphi)(re^{i\theta})\, d\theta 
        =
        \sup_{r\in (0,1)}(u\circ\varphi)(0)
        =
        (u\circ\varphi)(0).
    \end{align*}
    Since $\varphi(0)=i$, we have
    \begin{align*}
        (u\circ\varphi)(0)
        =
        u(0,1)
        =
        \frac{1}{\pi}\int_{\R}\frac{|F|(t)}{1+t^2}\, dt
        \lesssim 
        \|F\|_{L^p(\R,w)}
        \leq 
        \|F\|_{H^p(\R_{+}^{2},w)}
        <
        \infty,
    \end{align*}
    where the bound of the integral follows from formula (3.10) in \cite{CNO}. Since $F\circ\varphi\in H^1(\mathbb{D})$ and $(F\circ\varphi)(w_0)\neq 0$ for $w_0\coloneqq \varphi^{-1}(x,t)$, we can apply Theorem \ref{thm:thmII.4.1_Garnett}, obtaining
    \[
    \log{|F\circ\varphi(w_0)|}
    \leq 
    P_{\mathbb{D}}[\log{(F\circ\varphi)}](w_0).
    \]
    Composing with $\varphi^{-1}$
    and using Lemma \ref{lemma:equivalencePoissonIntegralsBis}, we have
    \[
    \log{|F(x,t)|}
    \leq 
    \mathbb{P}_{\mathbb{D}}[\log{|(F\circ\varphi)|}]\circ(\varphi^{-1})(x,t)
    =
    \mathbb{P}_{\R_{+}^{2}}[\log{|F|}](x,t).\qedhere
    \]
\end{proof}

\subsection{The Neumann integral}

Just as the Poisson integral is the standard solution to the Dirichlet problem in $\R_{+}^{2}$, the solution for the Neumann problem is the so-called Neumann integral \cite{ArmitageHalfSpace}.

\begin{definition}\label{def:neumannIntegral}
    If $f$ satisfies $\int_{\R}\frac{|f(x)|}{1+|x|}\, dx<\infty$, we define the Neumann integral of $f$ as
    \begin{equation}\label{eq:NeumannIntegral}
    u_{f,N}\left(y_1, y_2\right)
    \coloneqq 
    -\frac{1}{\pi} \int_{\mathbb{R}} \log \left(\frac{\sqrt{(y_1-x)^2+y_2^2}}{1+|x|}\right) f(x) \, dx, \quad (y_1, y_2) \in \R_{+}^2. 
    \end{equation}
    \end{definition}
    It is known  (\cite{CNO}) that $u_{f,N}$ is a well-defined harmonic function with gradient given pointwise by
    \begin{equation}\label{eq:gradientNeumannIntegral}
    \nabla u_{f,N}(y_1,y_2)
    =
    (-Q_{y_2}\ast f(y_1),-P_{y_2}\ast f(y_1)),
    \quad (y_1,y_2)\in \R_{+}^{2},    
    \end{equation}
    and hence, it is a solution to the Neumann problem in $\R_{+}^{2}$ with datum $f$.
    
For the estimates that we need, the key result to control $\mathcal{M}_{\alpha}(\nabla u_{f,N})$ is the following.

\begin{lemma}\label{thm:sparseDominationNonTangential}
    Let $f \in L^1(\mathbb{R})$, let $u_{f,N}$ be as in \eqref{eq:NeumannIntegral}. Then, there exists $\eta\in (0,1)$ such that for every $f\in L^1(\R)$ there is an $\eta$-sparse family $\mathcal{S}$ such that 
    \[
    \mathcal{M}_{\alpha}(\nabla u_{f,N})\lesssim 
    \mathcal{A}_{\mathcal{S}}f.
    \]
    As a consequence, the operator $f\mapsto \mathcal{M}_{\alpha}(\nabla u_{f,N})$ is bounded from $L^{p}(w)\to L^{p}(w)$ when $p\in (1,\infty)$ and $w\in A_p$, and from $L^1(w)\to L^{1,\infty}(w)$ if $w\in A_1$.
\end{lemma}

\begin{proof}
    This is contained in \cite{CNO}. For the first part, see pages 28-36. For the second part, see the density argument done in pages 36-37.
\end{proof}

\subsubsection{The class \texorpdfstring{$\mathcal{S}_{p}^{\mathcal{R}}$}{}}

In order to obtain norm estimates for $ \mathcal{M}_{\alpha}(\nabla u_{f,N})$ when we deal with $L^{p,1}(\Lambda,\nu)$-solvability, we need to introduce the following class of weights.  

\begin{definition}\label{def:SpR}
        We define $\mathcal{S}_p^{\mathcal{R}}$ as the class of ordered pairs of weights $(u,v)$ such that
        \begin{equation}\label{eq:twoWeightBound}
        \left\|\frac{\mathcal{A}_{\mathcal{S}}(fu)}{u}\right\|_{L^{p,\infty}(\R,v)}
        \lesssim 
        \|f\|_{L^{p,1}(\R,v)},\quad  \forall f\in L^{p,1}(\R,v),   
        \end{equation}
        for every sparse family $\mathcal{S}$. The implicit constant may depend on the sparseness constant $\eta\in (0,1)$.
\end{definition}

\begin{remark}\label{remark:SpFuerte}
We observe  that  if we define the class $\mathcal{S}_p$ by the condition 
\[
\left\|\frac{\mathcal{A}_{\mathcal{S}}(fu)}{u}\right\|_{L^{p}(\R,v)}
\lesssim 
\|f\|_{L^{p}(\R,v)},\quad  \forall f\in L^{p}(\R,v),  
\]
then it is easy to see that 
\[
(u,v)\in \mathcal{S}_p \quad\iff\quad  u^{-p}v\in A_p.
\]
However, the complete characterization of the class $\mathcal{S}_p^{\mathcal{R}}$ is a hard open problem in weighted theory.    
\end{remark}

Some concrete cases can be derived easily from previously known results. We state a few of them.

\begin{proposition}\label{prop:dualitySpR} Let $u,v$ be weights. Then:

\begin{enumerate}

\item $(u,v)\in \mathcal{S}_p^{\mathcal{R}}$ if and only if $(u^{-1}v,v)\in \mathcal{S}_{p'}^{\mathcal{R}}$.

\item $(1,v)\in \mathcal{S}_{p}^{\mathcal{R}}$ if and only if $u\in A_p^{\mathcal{R}}$.

\item $(u,u)\in \mathcal{S}_p^{\mathcal{R}}$ if and only if $u\in A_{p'}^{\mathcal R}$.

\end{enumerate}
\end{proposition}

\begin{proof} (1) follows by duality since sparse operators are self-adjoint. 

\noindent 
(2) It is known that $\mathcal{A}_{\mathcal{S}}\colon L^{p,1}(u)\to L^{p,\infty}(u)$ if $u\in A_{p}^{\mathcal{R}}$; see the proof of Corollary 3.2 in \cite{carroDomingoSalazar}. 

\noindent 
(3)  It follows from (1) and (2).\qedhere

\end{proof}

A more novel result is the following sufficient codition for $\mathcal{S}_p^{\mathcal{R}}$, that we will use to construct examples in Section \ref{sec:examplesLp1}. It generalizes items (2) and (3) in Proposition \ref{prop:dualitySpR}.

\begin{theorem}\label{thm:sufficientConditionSpR}
    Assume $u$ and $v$ are weights, with $v\in A_{\infty}$, such that
    \begin{equation}\label{eq:necessaryWeakConditionSpR}
    \sup_{Q}\displaystyle\frac{u(Q)\left(\frac{v}{u}\right)(Q)}{|Q|v(Q)}
    <
    \infty,    
    \end{equation}
    and such that
    \begin{equation}\label{eq:sufficientConditionSpR}
    \frac{v}{u}\in A_p(u)\text{ and } u\in A_{p'}^{\mathcal{R}}\left(\frac{v}{u}\right),\quad \text{ or }\quad  \frac{v}{u}\in A_p^{\mathcal{R}}(u) \text{ and } u\in A_{p'}\left(\frac{v}{u}\right).        
    \end{equation}
    Then $(u,v)\in \mathcal{S}_p^{\mathcal{R}}$.
\end{theorem}

\begin{proof}
    We consider the first case: $\frac{v}{u}\in A_p(u)$ and $u\in A_{p'}^{\mathcal{R}}(\frac{v}{u})$. The other case is settled in the same way. It is also enough to prove the bound for characteristic functions, so let $f=\mathds{1}_{E}$. We argue by duality. Let $g\in L^{p',1}(v)$ with $\|g\|_{L^{p',1}(v)}=1$. We have
    \[
        \int_{\R}\frac{\mathcal{A}_{\mathcal{S}}(fu)}{u}g(x)v(x)\, dx
        =
        \sum_{Q\in\mathcal{S}}\frac{u(E\cap Q)}{|Q|}\left(\int_{Q} g\frac{v}{u}\right).
    \]
    We use that $\frac{1}{|Q|}\lesssim \frac{v(Q)}{u(Q)\left(\frac{v}{u}\right)(Q)}$ to bound
    \[
    \sum_{Q\in\mathcal{S}}\frac{u(E\cap Q)}{|Q|}\left(\int_{Q} g\frac{v}{u}\right)
    \lesssim 
    \sum_{Q\in\mathcal{S}}\frac{u(E\cap Q)}{u(Q)}\frac{1}{\left(\frac{v}{u}\right)(Q)}\left(\int_{Q}g\frac{v}{u}\, dx\right)v(Q).
    \]
    
    We have $v\in A_{\infty}(\R)$ and hence
     \[
    v(Q)  \lesssim_{\eta}
    v(E_Q).
    \]
    Therefore, using this estimate and the disjointness of the sets $E_Q$, the above sum can be bounded (up to multiplicative constants) by
    \begin{align*}
       &\sum_{Q\in\mathcal{S}}\frac{u(E\cap Q)}{u(Q)}\frac{1}{\left(\frac{v}{u}\right)(Q)}\left(\int_{Q}g\frac{v}{u}\, dx\right)v(E_Q)
        \leq 
        \int_{\R}\mathcal{M}_{u}(\mathds{1}_E)\mathcal{M}_{\frac{v}{u}}(g)v(x)\, dx
          \\&\lesssim 
        \|\mathcal{M}_{u}(\mathds{1}_{E})\|_{L^{p,1}(v)}\|\mathcal{M}_{\frac{v}{u}}(g)\|_{L^{p',\infty}(v)}
      \lesssim 
        \|\mathcal{M}_{u}(\mathds{1}_{E})\|_{L^{p,1}(v)}\|g\|_{L^{p',1}(v)}
        \lesssim 
        \|f\|_{L^{p,1}(v)},
    \end{align*}
    having used in the last step that, if $\frac{v}{u}\in A_{p}(u)$, then $\frac{v}{u}\in A_{p-\varepsilon}(u)$, so by interpolation with the $L^{\infty}$-bound, $\mathcal{M}_{u}\colon L^{p,1}(v)\to L^{p,1}(v)$.\qedhere
\end{proof}

\section{The Dirichlet problem in \texorpdfstring{$L^{p,1}(\Lambda,\nu)$}{}}

\subsection{\texorpdfstring{$L^{p}(\Lambda,\nu)$}{}-solvability}
By Theorem \ref{thm:thm4.4_Kenig}, it is clear that,  if $p_{\Phi(\nu)}>1$,  the Dirichlet problem is not solvable in $L^{p_{\Phi(\nu)}}(\Lambda, \nu)$ because $\Phi(\nu)\not\in A_{p_{\Phi}(\nu)}(\R)$. Indeed, if it were solvable, by the properties of the $A_p$-classes we would have $\Phi(\nu)\in A_{p_{\Phi}(\nu)-\varepsilon}(\R)$, contradicting the definition of $p_{\Phi}(\nu)$ as infimum.

However, it may be solvable in a smaller space than $L^{p_{\Phi}(\nu)}(\Lambda, \nu)$, for example in $L^{p_{\Phi}(\nu),1}(\Lambda,\nu)$. Theorem \ref{thm:dirichletEndpointPesos} asserts that this occurs when $\Phi(\nu)\in A_{p_{\Phi}(\nu)}^{\mathcal{R}}(\R)$. Notice that this hypothesis makes sense: since $A_{p_{\Phi}(\nu)}\subsetneq A_{p_{\Phi}(\nu)}^{\mathcal{R}}$, it may be the case that $\Phi(\nu)\in A_{p_{\Phi}(\nu)}^{\mathcal{R}}$, as we will check in concrete examples.

Before proving Theorem \ref{thm:thm4.4_Kenig}, we compute the $L^{p,1}$-norm of the boundary datum.

\begin{lemma}\label{lemma:equivalenciaNormaLp1DirichletPesoNu}
Let $g\in L^{p,1}(\Lambda,\nu)$, with $p\in [1,\infty]$. Let $\mathcal{T}_{D}(g)\coloneqq g\circ\Phi$. Then
\[
\|g\|_{L^{p,1}(\Lambda,\nu)}
=
\|\mathcal{T}_{D}(g)\|_{L^{p,1}(\R,\Phi(\nu))}.
\]
\end{lemma}

\begin{proof}[Proof of Theorem \ref{thm:dirichletEndpointPesos}]
Let $p\coloneqq p_{\Phi}(\nu)$. Since $\mathcal{T}_{D}(g)\in L^{p,1}(\R,\Phi(\nu))$ with $\Phi(\nu)\in A_{p}^{\mathcal{R}}(\R)$, we have that  $\int_{\R}\frac{|f(x)|}{1+|x|}<\infty$  (\cite{CNO}). Therefore,
\[
u(y_1,y_2)
\coloneqq 
(P_{y_2}\ast \mathcal{T}_{D}(g))(y_1),\quad (y_1,y_2)\in\R_{+}^{2},
\]
is well-defined, harmonic in $\R_{+}^{2}$ and convergent non-tangentially to $\mathcal{T}_{D}(g)$ a.e. and, by Theorem \ref{thm:Stein_thm1Page197},
\[
\mathcal{M}_{\beta}u
\lesssim_{\beta}
\mathcal{M}_{hl}(\mathcal{T}_{D}g).
\] 
Thus, since $\Phi(\nu)\in A_{p}^{\mathcal{R}}$,
\[
\|\mathcal{M}_{\beta}(u)\|_{L^{p,\infty}(\R,\Phi(\nu))}
\lesssim
\|\mathcal{T}_{D}(g)\|_{L^{p,1}(\R,\Phi(\nu))}.
\]

Then $v\coloneqq u\circ\Phi^{-1}$ is harmonic in $\Omega$ and has non-tangential limit $\mathcal{T}_{D}(g)\circ\Phi^{-1}=g$, because by \eqref{eq:thm1.1_JerisonKenig} the non-tangential regions are preserved by $\Phi$ and $\Phi^{-1}$. Besides, by Theorem \ref{thm2.8_Kenig} and Remark \ref{generalizationThm2.8_Kenig}, 
\begin{equation}\label{equivalenceNonTangential1}
\|\mathcal{M}_{\alpha}(u)\|_{L^{p,\infty}(\R,\Phi(\nu))}
\simeq 
\|\mathcal{M}_{\beta}(v)\|_{L^{p,\infty}(\Lambda, \nu)}.    
\end{equation}
So
\[
\|\mathcal{M}_{\alpha}(v)\|_{L^{p,\infty}(\Lambda, \nu)}
\simeq 
\|\mathcal{M}_{\beta}(u)\|_{L^{p,\infty}(\R,\Phi(\nu))}
\lesssim 
\|\mathcal{T}_{D}(g)\|_{L^{p,1}(\R,\Phi(\nu))}
=
\|g\|_{L^{p,1}(\Lambda,\nu)}.\qedhere
\]
\end{proof}

\section{The Neumann Problem in \texorpdfstring{$L^p(\Lambda, \nu)$}{}}

In this section, we give the proof of Theorem \ref{thm:Lp_Solvability}.  The $L^p(w)$-solvability in $\R_{+}^{2}$ is the following.

\begin{theorem}[Theorem 1.3 in \cite{CNO}]\label{thm:thm1.3_CNO}
    Given $p\in (1,\infty)$, if $w \in A_p(\mathbb{R})$, then the Neumann problem in $\mathbb{R}_{+}^2$ is solvable in $X=L^p(\mathbb{R}, w)$ taking $Y=L^p(\mathbb{R}, w)$. In fact, if $f\in X$ is the datum, a solution is given by $u_{f,N}$.
\end{theorem}

\begin{theorem}\label{thm:page41_CNO}
    Let $g\in L_{loc}^{1}(\Lambda)$ be a function defined on $\Lambda$. Then $u$ is a solution to the Neumann problem in $\R_{+}^{2}$ with datum $(g\circ\Phi)|\Phi'|$, i.e.,
    \begin{equation}\label{eq:ecuacionNeumannSemiplano1}
    \begin{cases}
        \Delta u=0 & \text{ in }\R_{+}^{2}\\
        \partial_{\mathbf{n}}u=(g\circ\Phi)|\Phi'| & \text{ on } \partial\R_{+}^{2},
    \end{cases}    
    \end{equation}
    if and only if $v\coloneqq v\circ\Phi^{-1}$ is a solution to the Neumann problem on $\Omega$ with datum $g$, i.e., 
    \begin{equation}\label{eq:ecuacionNeumannDominio1}
    \begin{cases}
        \Delta v=0 & \text{ in } \Omega\\
        \partial_{\mathbf{n}}v=g & \text{ on } \Lambda=\partial\Omega.
    \end{cases}
    \end{equation}
\end{theorem}

\begin{proof}
    The argument provided in page 41 of \cite{CNO} justifies that if $u$ is a solution of \eqref{eq:ecuacionNeumannSemiplano1} then $v$ is a solution of \eqref{eq:ecuacionNeumannDominio1}. The converse can be justified in the same way using \eqref{eq:thm1.1_JerisonKenig}.
\end{proof}

    Given $g\in L^p(\Lambda, \nu)$, $p\in [1,\infty)$, we define 
    \begin{equation}\label{eq:NeumannDataHalfPlane}
    \mathcal{T}_{N}(g)
    \coloneqq 
    (g\circ\Phi)|\Phi'|.
    \end{equation}
    We have that 
     \begin{equation}\label{chv}
        \|g\|_{L^p(\Lambda,\nu)}
        =
        \|\mathcal{T}_{N}(g)\|_{L^p(\R,|\Phi'|^{-p}\Phi(\nu))}.
     \end{equation}

\begin{figure}
    \centering

    \begin{tikzpicture}[scale=1]
    \fill[cyan, opacity=0.3] (5.76, 0.657) -- (5.76, 3.217) -- (9.92, 3.217) -- (9.92, 1.297) -- (9.6, 1.937) -- (8.96, 1.297) -- (7.68, 0.657) -- (7.36, 1.297) -- (6.4, 1.937) -- (5.76, 0.657) -- cycle;
    \fill[cyan, opacity=0.3] (0, 0.657) -- (0, 3.217) -- (3.84, 3.217) -- (3.84, 0.657) -- (0, 0.657) -- (0, 0.657) -- cycle;
    \draw[black, thick, ->] (0, 0.657) -- (3.84, 0.657);
    \draw[black, thick, ->] (1.92, 0.017) -- (1.92, 3.217);
    \draw[black, thick, ->] (3.52, 2.577) .. controls (4.48, 3.217) and (5.44, 3.217) .. (6.08, 2.577);
    \draw[black, thick, ->] (5.76, 0.657) -- (6.4, 1.937) -- (7.36, 1.297) -- (7.68, 0.657) -- (8.96, 1.297) -- (9.6, 1.937) -- (9.92, 1.297);
    \node[anchor=center, font=\normalsize] at (0.439, 2.897) {$\mathbb{R}_{+}^{2}$};
    \node[anchor=center, font=\normalsize] at (9.6, 2.897) {$\Omega$};
    \node[anchor=center, font=\normalsize] at (4.8, 2.577) {$\Phi$};
    \node[anchor=center, font=\normalsize] at (2.56, 2.257) {$u$};
    \node[anchor=center, font=\normalsize] at (8, 2.257) {$v=u\circ\Phi^{-1}$};
    \node[anchor=center, font=\normalsize] at (8.439, 0.657) {$g(\xi)$};
    \node[anchor=center, font=\normalsize] at (2.904, 0.206) {$\mathcal{T}_{N}g(x)$};
\end{tikzpicture}

    \caption{Transference of solutions of the Neumann problem from $\R_{+}^{2}$ to $\Omega$.}
    \label{fig:transferenceSolutions}
\end{figure}
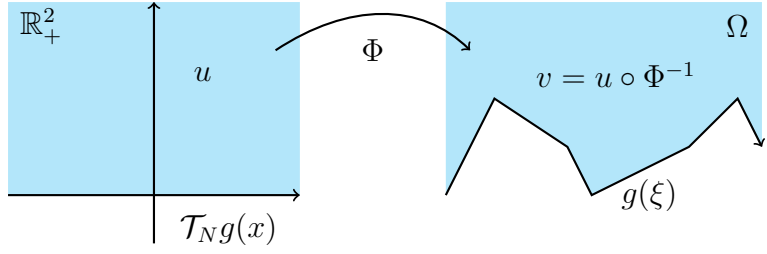

\subsection{Proof of Theorem \ref{thm:Lp_Solvability}}
        The proof follows the same argument as the one for Theorem 1.4 in \cite{CNO}. We just need to replace $|\Phi'|^{1-p}$ by the more general $|\Phi'|^{-p}\Phi(\nu)$, and use Theorem \ref{thm:analogoLema4.1CNO} to deal with the Hardy space norms. We give the details for completeness.

        Let $g\in L^p(\Lambda, \nu)$. Consider the following Neumann problem in $\R_{+}^{2}$:
        \begin{equation}\label{eq:problemaSemiplanoProblemaNeumannPesoNu}
            \begin{cases}
            \Delta u=0 & \text{ in } \R_{+}^{2},\\
            \partial_{\mathbf{n}}u= \mathcal{T}_{N}(g) & \text{ on }\partial \R_{+}^{2},\\ 
            \|\mathcal{M}_{\alpha}(\nabla u)\|_{L^p(\R,|\Phi'|^{-p}\, \Phi(\nu))}\lesssim \|\mathcal{T}_N(g)\|_{L^p(\R,|\Phi'|^{-p}\, \Phi(\nu))}. &
            \end{cases}
        \end{equation}
        By Theorem \ref{thm:thm1.3_CNO}, a solution is $u=u_{\mathcal{T}_N(g)}$ as defined in \eqref{eq:NeumannIntegral}, because $|\Phi'|^{-p}\Phi(\nu)\in A_p(\R)$. Let us see that
        \[
        v_{g}(z_1,z_2)
        \coloneqq 
        u_{\mathcal{T}_{N}(g)}\circ \Phi^{-1}(z_1,z_2),\quad (z_1,z_2)\in \Omega,
        \]
        is a solution to
        \begin{equation}\label{eq:problemaDominioProblemaNeumannPesoNu}
        \begin{cases}
            \Delta v=0 & \text{ in }\Omega, \\
            \partial_{\mathbf{n}}v=g & \text{ on }\partial\Omega=\Lambda,\\
            \|\mathcal{M}_{\beta}(\nabla v)\|_{L^p(\Lambda,\nu)}\lesssim \|g\|_{L^p(\Lambda,\nu)}.
        \end{cases}
        \end{equation}
        By Theorem \ref{thm:page41_CNO}, the only thing that remains to be proved is the non-tangential maximal estimate
        \begin{equation}\label{eq:estimacionPendienteDemNeumannNuLp}
        \|\mathcal{M}_{\beta}(\nabla v)\|_{L^p(\Lambda,\nu)}\lesssim \|g\|_{L^p(\Lambda,\nu)}.
        \end{equation}
 Since
        \[
        \|\mathcal{M}_{\alpha}(\nabla u)\|_{L^p(\R,|\Phi'|^{-p}\Phi(\nu))}
        \lesssim 
        \|\mathcal{T}_{N}(g)\|_{L^p(\R,|\Phi'|^{-p}\Phi(\nu)},
        \]
        to prove \eqref{eq:estimacionPendienteDemNeumannNuLp} it would suffice, by \eqref{chv},  to prove the following inequality:
        \[
        \|\mathcal{M}_{\beta}(\nabla v)\|_{L^p(\Lambda,\nu)}
        \lesssim 
        \|\mathcal{M}_{\alpha}(\nabla u)\|_{L^p(\R,|\Phi'|^{-p}\, \Phi(\nu))}.
        \]
        To use the theory of Hardy Spaces, we need to construct analytic functions.
        Following \cite{CNO}, we define $F$ as the conjugate gradient of $u$, that is,
        \[
        F(y_1,y_2)
        \coloneqq
        \partial_1 u(y_1,y_2)-i\partial_2 u(y_1,y_2),\quad (y_1,y_2)\in\R_{+}^{2},
        \]
        and $G$ as the conjugate gradient of $v$, 
        \[
        G(z_1,z_2)
        \coloneqq
        \partial_1 v(z_1,z_2)-i\partial_2 u(z_1,z_2),\quad (z_1,z_2)\in\Omega.
        \]
        Then 
        \[
        |F|=|\nabla u|\quad \text{ in }\R_{+}^{2},\quad |G|=|\nabla v|\quad \text{ in }\Omega,
        \]
        so
        \[
        \|\mathcal{M}_{\alpha}(\nabla u)\|_{L^p(\R,|\Phi'|^{-p}\, \Phi(\nu))}
        \simeq_{\alpha}
        \|F\|_{H^p(\R_{+}^{2},|\Phi'|^{-p}\, \Phi(\nu))}
        \]
        and
        \[
        \|\mathcal{M}_{\beta}(\nabla v)\|_{L^p(\Lambda,\nu)}
        \simeq_{\beta}
        \|G\|_{H^p(\Omega,\nu)}.
        \]
        So we have to prove that
        \begin{equation}\label{eq:goalProofLpSolvability}
        \|G\|_{H^p(\Omega,\nu)}
        \lesssim 
        \|F\|_{H^p(\R_{+}^{2},|\Phi'|^{-p}\Phi(\nu))}.
        \end{equation}
        It can be checked that
        \[
        G\circ\Phi=F\cdot \frac{1}{\Phi'}
        \quad \text{ on }\R_{+}^{2}.
        \]
        By Theorem \ref{thm:analogoLema4.1CNO}, which we can apply because $|\Phi'|^{-p}\Phi(\nu)\in A_p(\R)$, we have
        \[
        \|F\|_{H^p(\R_{+}^{2},|\Phi'|^{-p}\Phi(\nu))}
        \simeq 
        \left\|F\cdot \frac{1}{\Phi'}\right\|_{H^p(\R_{+}^{2},\Phi(\nu))}
        =
        \|G\circ \Phi\|_{H^p(\R_{+}^{2},\Phi(\nu))}.
        \]
        To finish we just apply Theorem \ref{thm2.8_Kenig} to obtain 
        \[
        \|G\circ \Phi\|_{H^p(\R_{+}^{2},\Phi(\nu))}
        \simeq 
        \|G\|_{H^p(\Omega,\nu)}.
        \]
        To sum up, $\|G\|_{H^p(\Omega, \nu)}\simeq \|F\|_{H^p(\R_{+}^{2},|\Phi'|^{-p}\Phi(\nu))}$, which in particular implies \eqref{eq:goalProofLpSolvability}.

        To prove uniqueness up to constants, let $g\in L^p(\Lambda,\nu)$ and assume $\Tilde{v}$ is a solution to \eqref{eq:problemaDominioProblemaNeumannPesoNu}. Then, by reasonings analogous to the ones we have just made, $\Tilde{u}\coloneqq \Tilde{v}\circ\Phi$ is a solution to \eqref{eq:problemaSemiplanoProblemaNeumannPesoNu}. Since $\mathcal{M}_{\alpha}(\nabla u)\in L^p(\R,|\Phi'|^{-p}\Phi(\nu))$ with $|\Phi'|^{-p}\Phi(\nu)\in A_p$, then we can apply Theorem 2.5 in \cite{transmissionCarro} and obtain that $\Tilde{u}=u_{\mathcal{T}_{N}(g),N}+C$ for some $C\in\R$. Therefore, composing with $\Phi^{-1}$, we have $\Tilde{v}=u_{\mathcal{T}_{N}(g),N}\circ\Phi^{-1}+C$. So the solution to \eqref{eq:problemaDominioProblemaNeumannPesoNu} is $u_{\mathcal{T}_{N}(g),N}\circ\Phi^{-1}$, up to additive constants.  

\subsection{The range of solvability}

By Theorem \ref{thm:Lp_Solvability}, we can solve the Neumann problem in $L^p(\Lambda,\nu)$ when $|\Phi'|^{-p}\Phi(\nu)\in A_p(\R)$, so the set of exponents $p$ for which the problem is solvable is given by
\[
R(\nu)
\coloneqq 
\{
q\in (1,\infty):|\Phi'|^{-q}\Phi(\nu)\in A_q(\R)
\}.
\]
Unlike the case $\nu=1$ explained in the introduction, it is not even clear that this set is non-empty. The following results aim to understand better its structure, which can be much more flexible than in the case of arc-length measure, where it is an interval of the form $(1,2+\varepsilon)$. Indeed,
\begin{itemize}
    \item There are examples of $\nu\in A_{\infty}(\Lambda)$ such that $R(\nu)=\varnothing$.
    \item When $R(\nu)\neq\varnothing$,  it is an interval that can have any size and any location. In particular, it can be at a positive distance from $p=1$, unlike the arc-length measure case $\nu\equiv 1$. 
\end{itemize}

We leave the examples for Section \ref{section:examplesDirichlet} and prove now the second item.

\begin{lemma}\label{lemma:lemma1RangeSolvability}
    Let $\nu\in A_{\infty}(\Lambda)$. If $R(\nu)\neq\varnothing$, then there are at least three points in $R(\nu)$. More specifically, given $p\in R(\nu)$, there exist $q_{-},q_{+}\in R(\nu)$ with $q_{-}<p<q_{+}$. 
\end{lemma}

\begin{proof}
    Let $p\in R(\nu)$. This means that $|\Phi'|^{-p}\Phi(\nu)\in A_p$, which can be rewritten as
    \[
    \nu(\Phi(x))|\Phi'(x)|^{1-p}\in A_p(\R).
    \]
    Therefore, by the Factorization Theorem,  
    \begin{equation}\label{eq:fact1_rangeSolvability}
    \exists u_0,u_1\in A_1(\R): \nu(\Phi(x))|\Phi'(x)|^{1-p}=u_0 u_1^{1-p}.    
    \end{equation}
    We want to check that there exists $q\neq p$ such that $q\in R(\nu)$, i.e., such that $\nu(\Phi(x))|\Phi'(x)|^{1-q}\in A_q$. Again, by the Factorization Theorem, this is equivalent to the existence of $v_0,v_1\in A_1(\R)$ such that
    \begin{equation}\label{eq:fact2_rangeSolvability}
        \nu(\Phi(x))|\Phi'(x)|^{1-p}=v_0 v_1^{1-q}.
    \end{equation}

    We try to use \eqref{eq:fact1_rangeSolvability} to obtain \eqref{eq:fact2_rangeSolvability}. We can write
    \begin{equation}\label{eq:ident1_rangeSolvability}
    \nu(\Phi(x))|\Phi'(x)|^{1-q}
    =
    \nu(\Phi(x))|\Phi'(x)|^{1-p}|\Phi'(x)|^{p-q}
    =
    u_0 u_1^{1-p}|\Phi'(x)|^{p-q}.    
    \end{equation}
    Since $|\Phi'(x)|\in A_2(\R)$, we can use once more the Factorization Theorem to obtain that there exists $\widehat{u}_0,\widehat{u}_1\in A_1(\R)$ such that
    \[
    |\Phi'(x)|
    =
    \widehat{u}_0\widehat{u}_1^{-1}.
    \]
    Substituting this into \eqref{eq:ident1_rangeSolvability}, we have
    \begin{equation}\label{eq:ident2_rangeSolvability}
    \nu(\Phi(x))|\Phi'(x)|^{1-q}
    =
    u_0 u_1^{1-p}(\widehat{u}_0\widehat{u}_1^{-1})^{p-q}
    =
    u_0 u_{1}^{1-p} \widehat{u}_{0}^{p-q} \widehat{u}_{1}^{q-p}.    
    \end{equation}

We have to distinguish the two cases $p-q>0$ and $p-q<0$. We just give the details of the case $p>q$, because the other one can be settled in the same way. So assume $p>q$. Then
    \begin{equation}\label{eq:case1_rangeSolvability}
        u_0 u_1^{1-p} \widehat{u}_{0}^{p-q}\widehat{u}_{1}^{q-p}
        =
        (u_0\widehat{u}_0^{p-q})(u_1^{1-p}\widehat{u}_1^{q-p})
        =
        (u_0\widehat{u}_0^{p-q})(u_1^{\frac{1-p}{1-q}}\widehat{u}_{1}^{\frac{q-p}{1-q}})^{1-q},
    \end{equation}
    so it is enough to prove that $u_0\widehat{u}_0^{p-q}\in A_1(\R)$ and $u_1^{\frac{1-p}{1-q}}\widehat{u}_{1}^{\frac{q-p}{1-q}}\in A_1(\R)$ for some value of $q<p$.

    We have that $u_0\widehat{u}_0^{p-q}\in A_1(\R)$ if  $p-q$ is small enough. Indeed, by Lemma \ref{lemma:improvingA1}, $u_0^{r}\in A_1(\R)$ for some $r>1$ small enough. So we can write
    \[
    u_0\widehat{u}_0^{p-q}
    =
    (u_0^{r})^{\frac{1}{r}}\widehat{u}_{0}^{p-q},
    \]
    which is in $A_1(\R)$ if $\frac{1}{r}+p-q\leq 1$ by Lemma \ref{lemma:improvingA1}.

    We also have that $u_1^{\frac{1-p}{1-q}}\widehat{u}_{1}^{\frac{q-p}{1-q}}\in A_1(\R)$ if $p-q$ is small enough. The reasoning is analogous using that  there exists $r>1$ sufficiently small such that  $u_1^r\in A_1(\R)$. 
\end{proof}

\begin{theorem}\label{lemma:lemma2RangeSolvability}
    Let $\nu\in A_{\infty}(\Lambda)$. If $R(\nu)\neq\varnothing$, then $R(\nu)$ is an open interval. \end{theorem}

\begin{proof}
    Let us see that it is a convex set. Indeed, assume $p_0,p_1\in R(\nu)$. That is,
    \[
    w_i\coloneqq \nu(\Phi(x))|\Phi'(x)|^{1-p_i}\in A_{p_i},\quad i=1,2. 
    \]
    By interpolation, if we fix $\theta\in [0,1]$ and define $p$ via the relation
    \begin{equation}\label{eq:convexLinearCombinationLemma}
    \frac{1}{p}=\frac{1-\theta}{p_0}+\frac{\theta}{p_1},    
    \end{equation}
    then $w\in A_p$, with
    \[
    w
    =
    (w_0^{\frac{1-\theta}{p_0}}w_{1}^{\frac{\theta}{p_1}})^p
    =
    \nu(\Phi)|\Phi'|^{1-p},
    \]
    and the result follows.
   \end{proof}

Therefore, if $R(\nu)\neq \varnothing$,  we can define
\begin{equation}\label{eq:p-}
p_{-}\coloneqq \inf R(\nu), \quad\mbox{and}\quad  p_{+}
    \coloneqq    \sup{R(\nu)}
\end{equation}

By the previous results, we know that $p_{-}\neq p_{+}$. Besides, $p_{-},p_{+}\not\in R(\nu)$. Indeed, if $p_{-}\in R(\nu)$, by Lemma \ref{lemma:lemma1RangeSolvability} there would exist some $q<p_{-}$ such that $q\in R(\nu)$, contradicting that $p_{-}$ is the infimum. The same works for $p_{+}$.

\subsection{A remark on duality}\label{sec:duality}

When $\nu\equiv 1$, the Dirichlet problem is solvable in $L^p(\Lambda,ds)$ if and only if $|\Phi'|\in A_p$. On the other hand, the Neumann problem is solvable in $L^{p'}(\Lambda,ds)$ when $|\Phi'|^{1-p'}\in A_{p'}$. Since $|\Phi'|\in A_p$ if and only if $|\Phi'|^{1-p'}\in A_{p'}$, if we knew optimality of these conditions, then we would have that the Dirichlet problem is solvable in $L^p$ if and only if the Neumann problem is solvable in $L^{p'}$. This is why the ranges of solvability are dual to each other: if $(p_{\Phi},\infty)$ is the Dirichlet range, the Neumann range is $(1,p_{\Phi}')$.

We can reason in an analogous way for a general measure $\nu$. Indeed, by Theorem \ref{thm:thm4.4_Kenig}, the Dirichlet problem in $L^{p}(\Lambda,\nu)$ is solvable if and only if $\Phi(\nu)\in A_p$. And $\Phi(\nu)\in A_p$ if and only if $\Phi(\nu)^{1-p'}\in A_{p'}$. It turns out that
\[
\Phi(\nu)^{1-p'}
=
|\Phi'|^{-p'}\Phi(\nu^{1-p'}),
\]
so if the Dirichlet problem is solvable in $L^p(\Lambda,\nu)$, then $|\Phi'|^{-p'}\Phi(\nu^{1-p'})\in A_{p'}$, and hence the Neumann problem is solvable in $L^{p'}(\Lambda,\nu^{1-p'})$.   For the reverse implication, we would need to know that if the Neumann problem is solvable in $L^p(\Lambda,\nu)$, then $|\Phi'|^{-p}\Phi(\nu)\in A_p$. {This is an open problem.}

\section{\texorpdfstring{The Neumann Problem in $H_{at}^{1}(\Lambda, \nu)$}{}}\label{sec:H1solvability}

We know from the $L^p$-case that if $v=u_{\mathcal{T}_{N}(g),N}\circ\Phi^{-1}$, then $v$ is a solution to the Neumann problem in $\Omega$ with boundary datum $g$. Therefore, it is the natural candidate also for the $H_{at}^{1}$-case as well.

\subsection{\texorpdfstring{$H^1(\R,w)$}{}-solvability in \texorpdfstring{$\R_{+}^{2}$}{}}

\begin{theorem}\label{thm:H1SolvabilityHalfPlane}
    Let $w\in A_{1}$. Then the Neumann problem is solvable in $X=H_{at}^{1}(\R,w)$ with $Y=L^1(\R,w)$. More explicitly, given $f\in H_{at}^1(\R,w)$, $u_{f,N}$ is a well-defined solution to the Neumann problem with datum $f$ that satisfies the non-tangential maximal estimate
    \[
    \|\mathcal{M}_{\alpha}(\nabla u_{f,N})\|_{L^1(\R,w)}
    \lesssim 
    \|f\|_{H_{at}^1(\R,w)}.
    \]
\end{theorem}

\begin{proof} We have to check three conditions.

\noindent
    \textbf{Harmonicity.} 
    As observed in page 29 of \cite{CNO}, $u_{f,N}$ is a well-defined harmonic function if $\int_{\R}\frac{|f(x)|}{1+|x|}\, dx<\infty$, and this is the case because functions in $f\in H_{at}^{1}(\R,w(x)dx)\subset L^1(\R,w(x)dx)$ with $w\in A_{1}$ satisfy that estimate by Lemma \ref{lemma:estimateNeumannIntegral}.

    \vspace{0.3cm}\noindent
    \textbf{Non-tangential convergence}. To justify the non-tangential convergence, notice that
    \[
    \nabla u_{f,N}(y_1,y_2)\cdot \mathbf{n}(x,0)
    =
    -\frac{\partial u_{f,N}}{\partial y_2}(y_1,y_2)
    =
    (P_{y_2}\ast f)(y_1), \quad (y_1,y_2)\in\R_{+}^{2},
    \]
    where the last equality comes from \eqref{eq:gradientNeumannIntegral}. Therefore,   
    \[
    \partial_{\mathbf{n}}u_{f,N}(x,0)
    =
    \lim_{(y_1,y_2)\triangleright (x,0)}
    (P_{y_2}\ast f)(y_1)
    =
    f(x),\quad  dx\text{-a.e.}
    \]
    by Theorem \ref{thm:Stein_thm1Page197}, which we can apply by Lemma \ref{lemma:estimateNeumannIntegral} because $f\in L^1(\R,w)$ with $w\in A_{1}$.

    \vspace{0.3cm}\noindent
    \textbf{Boundedness of non-tangential maximal operators.} Finally, given $f\in H_{at}^{1}(\R,w)$, we want to see that
    \[
    \|\mathcal{M}_{\alpha}(\nabla u_{f,N})\|_{L^1(\R,w)}
    \lesssim 
    \|f\|_{H_{at}^1(\R,w)},\quad f\in H_{at}^{1}(\R,w).
    \]
    To prove that an operator $T$ is bounded from $H_{at}^{1}\to L^1$ it is enough 
     to check that for any atom $a$, $\|Ta\|_{L^1}\lesssim 1$, and that $T$ is bounded from $L^1\to L^{1,\infty}$.
    
    Our operator is $f\mapsto \mathcal{M}_{\alpha}(\nabla u_{f,N})$. The $L^1(\R,w)\to L^{1,\infty}(\R,w)$ boundedness is true by Theorem \ref{thm:sparseDominationNonTangential}, so we just need to justify that, for any $(\R,w)$-atom $a$,
    \begin{equation}\label{eq:estimateAtoms}
        \|\mathcal{M}_{\alpha}(\nabla u_{a,N})\|_{L^1(\R,w)}\lesssim_{\alpha} 1.
    \end{equation}
    By \eqref{eq:gradientNeumannIntegral},
    \[
    \nabla u_{a,N}(y_1,y_2)
    =
    (-Q_{y_2}\ast a(y_1),-P_{y_2}\ast a(y_1)),
    \quad (y_1,y_2)\in \R_{+}^{2}.
    \]
    Since $a\in L_{c}^{\infty}$, it is known that $Q_{y}\ast a(x)=P_{y}\ast \mathcal{H}a(x)$, so we have
    \[
    |\nabla u_{a,N}(y_1,y_2)|
    = 
    |P_y\ast (a+i\mathcal{H}a)(x)|.
    \]
    Therefore, \eqref{eq:estimateAtoms} is justified by Theorem II.3.4 in
    \cite{GarciaCuerva}.  \qedhere
\end{proof}

\subsection{Norm of the boundary data}\label{section:H1normNeumannDatum}

In this subsection we prove the equivalence
\begin{equation}\label{eq:equivalenciaDatoH1}
    \|f\|_{H_{at}^{1}(\Lambda,\nu)}
    \simeq
    \|\mathcal{T}_{N}(f)\|_{H_{at}^{1}(\R,(\nu\circ\Phi))}
\end{equation}
between the boundary data of $\Omega$ and $\R_{+}^{2}$. 

\begin{lemma}\label{lemma:atomEquivalence1}
    A function $a\colon\Lambda\to\R$ is a $(\Lambda, \nu)$-atom if and only if $(a\circ\Phi)$ is a $(\R,\Phi(\nu))$-atom.
\end{lemma}

\begin{proof}
    The properties of $a$ being a $(\Lambda,\nu)$-atom are equivalent to those of $(a\circ\Phi)$ being a $(\R,\Phi(\nu))$-atom. Indeed, given $\xi_0\in\Lambda$ and $r>0$, we have that
        \[
        \xi\in B_{\Lambda}(\xi_0,r)
        \iff 
        \Phi^{-1}(\xi)\in \Phi^{-1}(B_{\Lambda}(\xi_0,r)),
        \]
        so
        \[
        \supp{a}\subset B_{\Lambda}(\xi_0,r)
        \iff 
        \supp{a\circ\Phi}
        \subset
        \Phi^{-1}(B_{\Lambda}(\xi_0,r)),
        \]
        with $\Phi^{-1}(B_{\Lambda}(\xi_0,r))=(\Phi^{-1}\circ\eta)(B(\eta^{-1}(\xi_0),r))$ an interval of $\R$.   Therefore, we need to prove that
    \[
    \|a\|_{L^{\infty}(\Lambda,\nu)}\leq \frac{1}{\nu(B_{\Lambda}(\xi_0,r))}
    \iff 
    \|a\circ\Phi\|_{L^{\infty}(\R,\Phi(\nu))}
    \leq 
    \frac{1}{\Phi(\nu)(\Phi^{-1}(B_{\Lambda}(\xi_0,r)))}.
    \]
    But this is true because, by a change of variables,
    \[
    \nu(B_{\Lambda}(\xi_0,r))
    =
    \Phi(\nu)(\Phi^{-1}(B_{\Lambda}(\xi_0,r))),
    \]
    and by Theorem \ref{thm:thm1.1_Kenig}, $\Phi\colon\R\to\Lambda$ is a bijective map that preserves sets of measure $0$, so
    \[
    \|a\circ \Phi\|_{L^\infty(\R)}
    =
    \|a\|_{L^\infty(\Lambda, \nu)}. 
    \]
    Finally, the mean $0$ property is direct because
    \[
    \int_{\R}(a\circ\Phi)(x)\, \Phi(\nu)(x)dx
    =
    \int_{\Lambda}a(\xi)\, \nu(\xi) d\sigma(\xi).\qedhere
    \]
\end{proof}

As a trivial consequence of Lemma \ref{lemma:atomEquivalence1}, we obtain:

\begin{corollary}\label{corol:firstEquivalence}
    With the above notation,
    \[
    f\in H_{at}^1(\Lambda,\nu)
    \iff
    f\circ\Phi\in H_{at}^1(\R,\Phi(\nu)),
    \]
    with
    \[
    \|f\|_{H_{at}^1(\Lambda,\nu)}
    =
    \|f\circ\Phi\|_{H_{at}^1(\R,\Phi(\nu))}.
    \]
\end{corollary}

As we explained above, we would like to obtain a relation between $\|f\|_{H_{at}^1(\Lambda,\nu)}$ and $\|\mathcal{T}_N(f)\|_{H_{at}^1(\R,\mu)}$ for a certain measure $\mu$, and we have proved the relation $\|f\|_{H_{at}^1(\Lambda,\nu)}=\|f\circ\Phi\|_{H_{at}^1(\R,\Phi(\nu))}$. Here is where the following theorem enters into play. 

\begin{theorem}\label{thm:equivalenceWeightedH1}
    Let $w\in A_{\infty}(\R^n)$. Then 
    \[
    H_{at}^1(\R^n,w)
    =
    \{fw^{-1}: f\in H_{at}^1(\R^n,dx)\},
    \]
    with equivalence of norms. In other words, the map $f\mapsto fw$ is an isomorphism between the weighted Hardy space $H_{at}^1(\R^n,w)$ and the classical (unweighted) Hardy space $H_{at}^1(\R^n,dx)$.
\end{theorem}

\begin{proof}
    See Theorem 2.25 in \cite{coifmanWeiss} or Theorem III 1.2 in \cite{GarciaCuerva}.
\end{proof}

\begin{corollary}\label{corol:equivalenciaNormaH1DatoNeumann}
    Let $\nu\in A_{\infty}(\Lambda)$ and suppose $\nu\circ\Phi\in A_{\infty}$. Then
    \[
    f\in H_{at}^1(\Lambda,\nu)
    \iff 
    \mathcal{T}_N(f)=(f\circ\Phi)|\Phi'|
    \in
    H_{at}^1(\R,(\nu\circ\Phi)),
    \]
    with the equivalence of norms \eqref{eq:equivalenciaDatoH1}.
\end{corollary}

\begin{proof}
By Corollary \ref{corol:firstEquivalence},
\[
f\in H_{at}^1(\Lambda,\nu)
\iff
f\circ\Phi\in H_{at}^1(\R,\Phi(\nu)),
\]
with equality of norms. Then, since $\Phi(\nu), (\nu\circ\Phi)\in A_{\infty}$, we can apply Theorem \ref{thm:equivalenceWeightedH1} twice to obtain
\[
f\circ \Phi\in H^1_{at}(\R,\Phi(\nu))
\iff
(f\circ \Phi)\Phi(\nu)\in H_{at}^1(\R,dx)\iff (f\circ\Phi)|\Phi'|
\in
H_{at}^1(\R,(\nu\circ\Phi))
\]
with equivalence of norms. \qedhere
\end{proof}

\subsection{Proof of Theorem \ref{thm:H1solvability}}

The proof is the same as that of Theorem \ref{thm:Lp_Solvability}, using instead the estimates developed in this section.

We just emphasize that, to prove uniqueness, Theorem 2.5 in \cite{transmissionCarro} works for $p=1$, because if $f\in L^1(\R,w)$ and $w\in A_1$, then $\int_{\R}\frac{|f(x)|}{1+|x|}\, dx<\infty$. This, together with the fact that $\int_{\R}w(x)\, dx=\infty$ for $w\in A_{\infty}$ and the definition of $A_1$-weight, makes their argument valid in this case too.

\section{The Neumann Problem in \texorpdfstring{$L^{p_{\pm},1}(\Lambda,\nu)$}{}}

The goal of this section is to prove Theorem \ref{thm:Lp1_Solvability} following the strategy of Theorem 1.5 in \cite{CNO}. However, new difficulties arise in this more general context. First, the problem that we need to study in $\R_{+}^{2}$ gives rise to the two-weight estimate \eqref{eq:twoWeightBound}. Second, the transference of the non-tangential maximal estimate from $\R_{+}^{2}$ to $\Omega$ is more subtle in this case, and will make use of the property about logarithms stated in Lemma \ref{thm:teoremaGarnettAdaptado}.

\subsection{Equivalent problem in \texorpdfstring{$\R_{+}^{2}$}{}}

Let $g\in L^{p,1}(\Lambda, \nu)$ and let $\mathcal{T}_{N}(g)=(g\circ\Phi)|\Phi'|$ be the corresponding Neumann datum in $\R_{+}^{2}$. By a change of variables,
\[
\|g\|_{L^{p,1}(\Lambda, \nu)}
=
\|g\circ\Phi\|_{L^{p,1}(\R,\Phi(\nu))}
=
\left\|
\frac{\mathcal{T}_{N}(g)}{|\Phi'|}
\right\|_{L^{p,1}(\R,\Phi(\nu))}.
\]
Therefore, if we want to pose a problem in $\R_{+}^{2}$ that allows us to solve the $L^{p,1}(\Lambda,\nu)$-problem in $\Omega$ by composing with a conformal mapping, the space $X$ from which we must take the boundary data for the Neumann problem is $\mathcal{L}^{p,1}(|\Phi'|,\Phi(\nu))$, formed by those $f$ with
\[
\|f\|_{\mathcal{L}^{p,1}(|\Phi'|,\Phi(\nu))}
\coloneqq 
\left\|
\frac{f}{|\Phi'|}
\right\|_{L^{p,1}(\R,\Phi(\nu))}
<
\infty.
\]
And the space $Y$ that we choose is $\mathcal{L}^{p,\infty}(|\Phi'|,\Phi(\nu))$, formed by those $h$ with
\[
\|h\|_{\mathcal{L}^{p,\infty}(|\Phi'|,\Phi(\nu))}
\coloneqq 
\left\|
\frac{h}{|\Phi'|}
\right\|_{L^{p,\infty}(\R,\Phi(\nu))}
<
\infty.
\]

For the Neumann integral of $f$ to be well-defined, we need $\int_{\R}\frac{|f(x)|}{1+|x|}\, dx<\infty$. By Hölder,
\begin{equation}\label{eq:wellDefinedEndpointHolder}
\int_{\R}\frac{|f(x)|}{1+|x|}\, dx
\lesssim
\|f\|_{\mathcal{L}^{p,1}(\R,(|\Phi'|,\Phi(\nu)))}
\left\|\frac{(\nu\circ\Phi)}{1+|\cdot|}\right\|_{L^{p',\infty}(\R,\Phi(\nu))}.       
\end{equation}
So it suffices to have
\begin{equation}\label{eq:finiteNormWeights}
    \left\|\frac{(\nu\circ\Phi)}{1+|\cdot|}\right\|_{L^{p',\infty}(\R,\Phi(\nu))}
    <
    \infty.
\end{equation}

\begin{lemma} \label{lemma:wellDefined_Sp+R} Let $p\in (p_-, p_+]$.  If $(|\Phi'|,\Phi(\nu))\in \mathcal{S}_{p}^{\mathcal{R}}$, then   $\int_{\R}\frac{|f(x)|}{1+|x|}\, dx<\infty$ for $f\in \mathcal{L}^{p,1}(|\Phi'|,\Phi(\nu))$.
\end{lemma}

\begin{proof} 

Let $(u,v):=(|\Phi'|,\Phi(\nu))$. We check \eqref{eq:finiteNormWeights}. Since $(u,v)\in \mathcal{S}_{p}^{\mathcal{R}}$, then $(u^{-1}v,v)\in \mathcal{S}_{p'}^{\mathcal{R}}$ by Proposition \ref{prop:dualitySpR}. Hence, 
\begin{equation}\label{eq:dualEstimate}
\left\|
\frac{\mathcal{A}_{\mathcal{S}}(h\cdot u^{-1}v)}{u^{-1}v}
\right\|_{L^{p',\infty}(\R,v)}
\lesssim 
\|h\|_{L^{p',1}(\R,v)}.    
\end{equation}

Now, we have that 
\[
\frac{1}{1+|x|}
\simeq 
\mathcal{M}_{hl}(\mathds{1}_{[-1,1]})(x),\quad x\in\R,
\]
and thus,  if we choose $h\coloneqq \mathds{1}_{[-1,1]}uv^{-1}$, we have that
\[
\left\|\frac{uv^{-1}}{1+|\cdot|}\right\|_{L^{p',\infty}(v)}
\simeq
\left\|
\frac{\mathcal{M}_{hl}(h\cdot u^{-1}v)}{u^{-1}v}
\right\|_{L^{p',\infty}(\R,v)}
\lesssim 
\left\|
\frac{\mathcal{A}_{\mathcal{S}}(h\cdot u^{-1}v)}{u^{-1}v}
\right\|_{L^{p',\infty}(\R,v)}
\lesssim 
\|h\|_{L^{p',1}(\R,v)}.
\]
Consequently, it remains to prove that, under our hypothesis,  $h\in L^{p',1}(v)$. Now, for every  $p\in (p_{-},p_{+})$, we have that $vu^{-p}\in A_p(\R)$ and hence the dual weight $v^{1-p'}u^{p'}\in L^1_{loc}$. Therefore,  $h\in L^{p'}(v)$ for every  $p\in (p_{-},p_{+})$, and thus $h\in L^{p', 1}(v)$. Finally, since $p_{+}>p$ for $p\in (p_{-},p_{+})$, then $p_{+}^{\prime}<p'$ and hence, since $h$ has compact support, 
$h\in L^{p_{+}^{\prime}, 1}(v)$, and the result follows.
\end{proof}

\begin{remark}
In the case $p=p_-$, \eqref{eq:finiteNormWeights} has to be imposed in order to have a well-defined Neumann integral. This condition is equivalent to the following inequality for every measurable set $E$: 
\begin{equation}\label{eq:conditionForWellDefined}
 \int_E \frac{|\Phi'(x)|}{1+|x|} dx \lesssim \left(\Phi(\nu)(E)\right)^{1/p_-'}.
\end{equation}
\end{remark}

\subsection{Proof of Theorem \ref{thm:Lp1_Solvability}} We assume that $p_{+}<\infty$ and $p_{-}>1$. In this way, we can treat both cases at the same time, taking $\bm{p}\in \{p_{-},p_{+}\}$.

Let $g\in L^{\bm{p},1}(\Lambda, \nu)$. We have to solve the following problem in $\Omega$
\[
\begin{cases}
    \Delta v=0 & \text{in }\Omega,\\
    \partial_{\mathbf{n}}v=0 & \text{on }\partial \Omega=\Lambda,\\
    \|\mathcal{M}_{\alpha}(\nabla v)\|_{L^{\bm{p},\infty}(\Lambda, \nu)}
    \lesssim 
    \|g\|_{L^{\bm{p},1}(\Lambda, \nu)}.
\end{cases}
\]
To do so, consider the problem  in $\R_{+}^{2}$
\begin{equation}\label{eq:equivalentProblemEndpointCase}
    \begin{cases}
    \Delta u=0 & \text{ in }\R_{+}^{2}\\
    -\partial_y u=\mathcal{T}_{N}(g) & \text{ on }\R_{+}^{2}\\
    \|\mathcal{M}(\nabla u)\|_{\mathcal{L}^{p{+},\infty}(|\Phi'|,\Phi(\nu))}\lesssim \|\mathcal{T}_{N}(g)\|_{\mathcal{L}^{p{+},1}(|\Phi'|,\Phi(\nu))}.
    \end{cases}
\end{equation}
 Since $(|\Phi'|,\Phi(\nu))\in \mathcal{S}_{\bm{p}}^{\mathcal{R}}$, a solution is given by  $u=u_{\mathcal{T}_{N}(g)}$ which is well-defined as a consequence of \eqref{eq:finiteNormWeights} (it is assumed for $\bm{p}=p_{-}$, and for $\bm{p}=p_{+}$ it follows from Lemma \ref{lemma:wellDefined_Sp+R}). Therefore, we have
\[
\|\mathcal{M}_{\alpha}(\nabla u)\|_{\mathcal{L}^{\bm{p},\infty}(\R,(|\Phi'|,\Phi(\nu)))}
\lesssim 
\|\mathcal{T}_{N}(g)\|_{\mathcal{L}^{\bm{p},1}(\R,(|\Phi'|,\Phi(\nu)))}
=
\|g\|_{L^{\bm{p},1}(\Lambda,\nu)}.
\]
We know that $v=u\circ\Phi^{-1}$ is a solution to the Neumann problem with datum $g$ on $\Omega$, so the only thing we have to prove is the non-tangential maximal estimate
\[
\|\mathcal{M}_{\alpha}(\nabla v)\|_{L^{\bm{p},\infty}(\Lambda, \nu)}
\lesssim 
\|g\|_{L^{\bm{p},1}(\Lambda, \nu)}.
\]
Therefore, it is enough to prove that 
\[
\|\mathcal{M}_{\alpha}(\nabla v)\|_{L^{\bm{p},\infty}(\Lambda,\nu)}
\lesssim 
\|\mathcal{M}_{\alpha}(\nabla u)\|_{\mathcal{L}^{\bm{p},\infty}(\R,(|\Phi'|,\Phi(\nu)))}.
\]
First, let us rewrite the left-hand side. As usual, let 
\[
F\coloneqq \partial_1 u-i\partial_2 u
\quad\text{ in }\R_{+}^{2}
\quad
\text{ and }
\quad 
G\coloneqq \partial_1 v-i\partial_2 v
\quad \text{ in }\Omega.
\]
Then $F$ and $G$ are holomorphic in $\R_{+}^{2}$ and $\Omega$ respectively, and they satisfy
\[
|F|=|\nabla u|
\quad\text{ in }\R_{+}^{2}
\quad
\text{ and }
\quad 
|G|=|\nabla v|
\quad\text{ in }\Omega
\]
and
\[
G\circ\Phi=\frac{F}{\Phi'}\quad \text{ in }\R_{+}^{2}.
\]
Then, by Theorem \ref{thm2.8_Kenig},
\[
\|\mathcal{M}_{\alpha}(\nabla v)\|_{L^{\bm{p},\infty}(\Lambda,\nu)}
=
\|\mathcal{M}_{\alpha}(G)\|_{L^{\bm{p},\infty}(\Lambda,\nu)}
\simeq 
\|\mathcal{M}_{\alpha}(G\circ\Phi)\|_{L^{\bm{p},\infty}(\R,\Phi(\nu))}
=
\left\|\mathcal{M}_{\alpha}\left(\frac{F}{\Phi'}\right)\right\|_{L^{\bm{p},\infty}(\R,\Phi(\nu))}.
\]
So the inequality that we need to prove is
\begin{equation}\label{eq:neededInequalityEndpoint}
\left\|\mathcal{M}_{\alpha}\left(\frac{F}{\Phi'}\right)\right\|_{L^{\bm{p},\infty}(\R,\Phi(\nu))}
\lesssim 
\left\|\frac{\mathcal{M}_{\alpha}(F)}{\Phi'}\right\|_{L^{\bm{p},\infty}(\R,\Phi(\nu))}.    
\end{equation}

To do so, we use an approximation argument. Since $\mathcal{T}_{N}(g)\in \mathcal{L}^{\bm{p},1}(\R,\Phi(\nu))<\infty$, then $\frac{\mathcal{T}_N(g)}{|\Phi'|}\in L^{\bm{p},1}(\R,\Phi(\nu))$. By Lemma 2.2 (b) in \cite{CNO}, there exists $\{g_m\}\subset C_c(\R)\cap L^{\bm{p},1}(\R,\Phi(\nu))$ such that 
\[
\left\|
\frac{\mathcal{T}_N(g)}{|\Phi'|}-g_m
\right\|_{L^{\bm{p},1}(\R,\Phi(\nu))}
\xrightarrow{m\to \infty}
0.
\]
So $f_m\coloneqq g_m\cdot |\Phi'|\in \mathcal{L}^{\bm{p},\infty}(\R,(|\Phi'|,\Phi(\nu)))$ satisfies
\[
\|\mathcal{T}_N(g)-f_m\|_{\mathcal{L}^{\bm{p},1}(\R,(|\Phi'|,\Phi(\nu)))}
\xrightarrow{m\to\infty}
0.
\]

Let $u_m\coloneqq u_{f_m,N}$. Notice that $f_m\in L^p(\R,|\Phi'|^{-p}\Phi(\nu))$ for any $p\in (1,\infty)$, because
\[
\int_{\R}|f_m|^p|\Phi'|^{-p}\Phi(\nu)\, dx 
=
\int_{\R}|\Phi'|^{p}|g_m|^p|\Phi'|^{-p}\Phi(\nu)\, dx
\leq 
\|g_m\|_{\infty}^p\int_{\supp{g_m}}\Phi(\nu)\, dx
<\infty. 
\]
Besides, by definition of $p_{-}$ and $p_{+}$, $|\Phi'|^{-p}\Phi(\nu)\in A_p(\R)$ for every $p\in (p_{-},p_{+})$. Therefore, we have
\[
\|\mathcal{M}_{\alpha}(\nabla u_m)\|_{L^p(\R,|\Phi'|^{-p}\Phi(\nu))}
\lesssim 
\|f_m\|_{L^p(\R,|\Phi'|^{-p}\Phi(\nu))},
\quad \forall p\in (p_{-},p_{+}).
\]

Let $F_m\coloneqq \partial_{1}u_m-i\partial_2 u_m$, $m\in\N$. The previous inequality tells us that
\[
\|F_m\|_{H^p(\R_{+}^{2},|\Phi'|^{-p}\Phi(\nu))}
\lesssim 
\|f_m\|_{L^p(\R,|\Phi'|^{-p}\Phi(\nu))}
=
\|g_m\|_{L^p(\R,\Phi(\nu))}.
\]
Since $F_m\in H^p(\R_{+}^{2},|\Phi'|^{-p}\Phi(\nu))$ with $|\Phi'|^{-p}\Phi(\nu)\in A_{\infty}$, by Theorem \ref{thm:analogoLema4.1CNO} we obtain $\frac{F_m}{\Phi'}\in H^p(\R_{+}^{2},\Phi(\nu))$ with
\[
\|F_m\|_{H^p(\R_{+}^{2},|\Phi'|^{-p}\Phi(\nu))}
\simeq 
\left\|
\frac{F_m}{\Phi'}
\right\|_{H^p(\R_{+}^{2},\Phi(\nu))}.
\]
Denote by $h_m$ the boundary values of $F_m$. Since $F_m\in H^p(\R_{+}^{2},|\Phi'|^{-p}\Phi(\nu))$ with $|\Phi'|^{-p}\Phi(\nu)\in A_p(\R)$ for any $p\in (p_{-},p_{+})$, then Lemma \ref{thm:teoremaGarnettAdaptado} gives us that
\[
\log{|F_m(y_1,y_2)|}
\leq 
P_{y_2}\ast \log|h_m|(y_1),\quad (y_1,y_2)\in \R_{+}^{2}.
\]
Since $\log{|\Phi'(y_1,y_2)|}=P_{y_2}\ast \log{|\Phi'|}(y_1)$ (see for example \cite{AE_FBY}), we have that
\[
\log{|F_m(y_1,y_2)|}-\log{|\Phi'(y_1,y_2)|}
\leq 
P_{y_2}\ast \log{|h_m|}(y_1)-P_{y_{2}}\ast \log{|\Phi'|}(y_1).
\]
That is,
\[
\log{\left|\frac{F_m(y_1,y_2)}{\Phi'(y_1,y_2)}\right|}
\leq 
P_{y_2}\ast \log{\left|\frac{h_m}{\Phi'}\right|}(y_1).
\]
Multiplying by $a>0$, 
\[
\log{\left|\frac{F_m(y_1,y_2)}{\Phi'(y_1,y_2)}\right|^{a}}
\leq 
P_{y_2}\ast \log{\left|\frac{h_m}{\Phi'}\right|^{a}}(y_1).
\]
Taking exponentials and using Jensen's inequality, we obtain
\begin{align*}
    \left|\frac{F_m(y_1,y_2)}{\Phi'(y_1,y_2)}\right|^{a}
    \leq 
    e^{P_{y_2}\ast \log{|\frac{h_m}{\Phi'}|^{a}}(y_1)}
    \leq 
    (P_{y_2}\ast e^{\log{|\frac{h_m}{\Phi'}|^{a}}})(y_1)
    =
    \left(P_{y_2}\ast \left|\frac{h_m}{\Phi'}\right|^{a}\right)(y_1).
\end{align*}
So
\[
\left|
\frac{F_m(y_1,y_2)}{\Phi'(y_1,y_2)}
\right|
\leq 
\left(P_{y_2}\ast \left|\frac{F_m}{\Phi'}\right|^{a}(y_1)\right)^{\frac{1}{a}},
\]
which by Theorem \ref{thm:Stein_thm1Page197} implies 
\[
\mathcal{M}_{\alpha}\left(\frac{F_m}{\Phi'}\right)
\leq 
\left(\mathcal{M}_{\alpha}\left(P_{y_2}\ast \left|\frac{h_m}{\Phi'}\right|^{a}\right)\right)^{\frac{1}{a}}
\lesssim 
\left(\mathcal{M}_{hl}\left(\left|\frac{h_m}{\Phi'}\right|^{a}\right)\right)^{\frac{1}{a}}.
\]
Therefore,
\begin{align*}
    \left\|\frac{F_m}{\Phi'}\right\|_{H^{\bm{p},\infty}(\R_{+}^{2},\Phi(\nu))}
    \lesssim 
    \left\|
    \mathcal{M}_{hl}\left(\left|\frac{h_m}{\Phi'}\right|^{a}\right)^{\frac{1}{a}}
    \right\|_{L^{\bm{p},\infty}(\R,\Phi(\nu))}
    =
    \left\|\mathcal{M}_{hl}\left(\left|\frac{h_m}{\Phi'}\right|^{a}\right)\right\|_{L^{\frac{\bm{p}}{a},\infty}(\R,\Phi(\nu))}^{\frac{1}{a}}.
\end{align*}
We choose $a>0$ sufficiently small so that $q\coloneqq \frac{\bm{p}}{a}$ is sufficiently big so that $\Phi(\nu)\in A_q(\R)$. Then, we have $\mathcal{M}_{hl}\colon L^{q,\infty}(\Phi(\nu))\to L^{q,\infty}(\Phi(\nu))$. Therefore, the above can be estimated by
\begin{equation}\label{eq:referenciaParaElPasoLímite}
    \left\|\frac{F_m}{\Phi'}\right\|_{H^{\bm{p},\infty}(\R_{+}^{2},\Phi(\nu))}
    \lesssim 
    \left\|
    \left|\frac{h_m}{\Phi'}\right|^{a}
    \right\|_{L^{\frac{\bm{p}}{a},\infty}(\R,\Phi(\nu))}^{\frac{1}{a}}
    =
    \left\|\frac{h_m}{\Phi'}\right\|_{L^{\bm{p},\infty}(\R,\Phi(\nu))}
    \leq 
    \left\|\frac{\mathcal{M}_{\alpha}(F_m)}{|\Phi'|}\right\|_{L^{\bm{p},\infty}(\R,\Phi(\nu))}
\end{equation}
obtaining inequality \eqref{eq:neededInequalityEndpoint} for the approximations $F_m$.

The limiting argument now follows as in pages 44 and 45 of \cite{CNO} with the obvious changes.

\section{The Regularity Problem}\label{section:regularity}

In this section, we show how to adapt the results for the Neumann problem to the Regularity problem. As usual, we begin with the case of the upper half-plane.

\subsection{The regularity problem in \texorpdfstring{$\R_{+}^{2}$}{}}

\begin{lemma}\label{lemma:calculoGradienteRegularidad}
    Let $f\in L^1_{loc}(\R)$ with weak derivative $f'$ such that
     \begin{equation}\label{eq:integrabilityWeakDerivative}
    \int_{\R}\frac{|f(x)|}{1+|x|^2}\, dx<\infty\quad{and}\quad  \int_{\R}\frac{|f'(x)|}{1+|x|}\, dx<\infty,
    \end{equation}
    and let $u_{f,D}(y_1,y_2) \coloneqq (P_{y_2}\ast f)(y_1)$. 
    Then, we have that 
    \begin{equation}\label{eq:gradientPoisson}
    \nabla u_{f,D}(y_1,y_2)
    =
    \left(-P_{y_2}\ast f'(y_1), Q_{y_2}\ast f'(y_1)\right), \quad\forall (y_1,y_2)\in \R_{+}^{2}.        
    \end{equation}
\end{lemma}

\begin{proof} The proof follows by standard distributional computations. 
\end{proof} 

As a direct consequence of \eqref{eq:gradientNeumannIntegral} and \eqref{eq:gradientPoisson}, we obtain the following relation between the gradient of the Poisson integral and that of the Neumann integral. 

\begin{corollary}\label{corol:relacionNeumannRegularidad}
    Let $f$ be as in Lemma \ref{lemma:calculoGradienteRegularidad}. Then
    \begin{equation}\label{eq:relacionGradientesNeumannRegularidad}
        |\nabla u_{f,D}(y_1,y_2)|
        =
        |\nabla u_{f',N}(y_1,y_2)|,\quad \forall (y_1,y_2)\in\R_{+}^{2}.
    \end{equation}
\end{corollary}

Using \eqref{eq:relacionGradientesNeumannRegularidad}, the estimates for the Neumann problem can be transferred to the Regularity problem. 

\begin{lemma}\label{lemma:solvabilityRegularityProblemHalfPlane}
    Consider one of the following pairs $(X,Y)$ of spaces:
    \begin{itemize}
        \item $X=L^p(w)$, $p\in (1,\infty)$, $w\in A_p(\R)$, with $Y=L^p(w)$.
        \item $X=\mathcal{L}^{p,1}(u,v)$, $(u,v)\in \mathcal{S}_p^{\mathcal{R}}$ and $\|(uv^{-1})(1+|\cdot|)^{-1}\|_{L^{p',\infty}(v)}<\infty$,
        with $Y=\mathcal{L}^{p,\infty}(u,v)$. 
        \item $X=H_{at}^{1}(\R,w)$, $w\in A_1(\R)$, with $Y=L^1(\R,w)$.
    \end{itemize}
    If $f$ is as in Lemma \ref{lemma:calculoGradienteRegularidad}, then
    \[
    \|\mathcal{M}_{\alpha}(\nabla u_{f,D})\|_{Y}
    \lesssim_{\alpha}
    \|f'\|_{X}.
    \]
\end{lemma}

\begin{proof}
For the first case, we can apply formula (3.16) in page 36 of \cite{CNO} to $u_{f',N}$ and obtain
\begin{equation}\label{eq:maximalEstimateXY}
\|\mathcal{M}_{\alpha}(\nabla u_{f',N})\|_{Y}
\lesssim 
\|f'\|_{X}.    
\end{equation}
Since $f'\in X$, $f'$ satisfies \eqref{eq:integrabilityWeakDerivative} by Lemma 3.1 in \cite{CNO}, so we can apply Corollary \ref{corol:relacionNeumannRegularidad} and obtain:
\[
\|\mathcal{M}_{\alpha}(\nabla u_{f,D})\|_{Y}
=
\|\mathcal{M}_{\alpha}(\nabla u_{f',N})\|_{Y}
\lesssim 
\|f'\|_{X}.
\]

For the $\mathcal{L}^{p,1}(u,v)$-case, estimate \eqref{eq:maximalEstimateXY} is obtained by definition, and estimate \eqref{eq:integrabilityWeakDerivative} follows from the hypothesis $\|(uv^{-1})(1+|\cdot|)^{-1}\|_{L^{p',\infty}(v)}<\infty$ using \eqref{eq:wellDefinedEndpointHolder}.

For the remaining case $X=H_{at}^{1}(\R,w)$ with $w\in A_1$, we use Theorem \ref{thm:H1SolvabilityHalfPlane} in an analogous way. \qedhere
\end{proof}

These results address the Regularity problem in $\R_{+}^2$. To pass to $\Omega$, some care must be taken to deal with the derivatives of the boundary datum. First, we need some technical remarks.

\subsection{Some technical facts about conformal mapping}

Since $\Phi(x)\in\Lambda$, it must be of the form 
\[
\Phi(x)=(\Phi_1(x),\gamma(\Phi_1(x))),\quad x\in\R
\]
for some $\Phi_1\colon \R\to\R$. Hence, by the Chain Rule, 
\[
\Phi'(x)=\Phi_1'(x)\cdot (1,\gamma'(\Phi_1(x))).
\]
Since $\Phi'(x)\neq 0$ $dx$-a.e., it must be $\Phi_1'(x)\neq 0$ $dx$-a.e. in $\R$. Since $|\operatorname{arg}\Phi'(x)|\leq\arctan{L}<\frac{\pi}{2}$, then $\operatorname{Re}(\Phi'(x))=\Phi_1'(x)>0$ at every point where $\Phi'(x)$ exists and is different from $0$. In particular, 
\begin{equation}\label{eq:positiveDerivativePhi1}
\Phi_1'(x)>0 \text{ for a.e. } x\in\R. 
\end{equation}
Since it is a bijection from $\R\to\R$, it is strictly  increasing.
 
Besides, by the Inverse Function Theorem on $\R$, we have that $\Phi_1^{-1}$ has derivative $dx$-a.e., with
\begin{equation}\label{eq:derivativeInversePhi1}
(\Phi_1^{-1})'(x)
=
\frac{1}{\Phi_1'(\Phi_1^{-1}(x))}
>
0,\quad dx\text{-a.e.}    
\end{equation}
In summary, we have:

\begin{proposition}\label{prop:observacionPhi1Prima}
    If $\Phi\colon \R\to\Lambda=\partial\Omega$ is given by the components $\Phi=(\Phi_1,\Phi_2)$, then $\Phi_1\colon \R\to\R$ is an absolutely continuous homeomorphism from $\R\to\R$, differentiable a.e. with positive derivative. Its inverse $\Phi_{1}^{-1}$ is differentiable a.e. with derivative given by \eqref{eq:derivativeInversePhi1}.
\end{proposition}

\subsection{Derivative of a function defined over a Lipschitz curve}

To solve the Regularity problem in $\Omega$, we need to deal with estimate \eqref{eq:regularityEstimate}, which involves 
the weak derivative $f'$ of the Dirichlet datum $f$ appears. This datum $f$ is a function defined on the boundary, $f\colon \Lambda\to \R$, so we must clarify what we understand as the weak derivative of a function defined on a Lipschitz curve. This is defined in some articles like \cite{VerchotaLayer}, but here we modify slightly the class of test functions in order to make rigorous some explicit computations that we have to perform. First, we introduce some notation.

\begin{definition}
Let $g\colon \Lambda\to \R$ be a measurable function. We say that $g$ is locally integrable, and write $g\in L_{loc}^{1}(\Lambda)$ if $(g\circ\eta)(t)=g(t,\gamma(t))\in L_{loc}^{1}(\R,dt)$.
\end{definition}

We are going to use functions of bounded variation as test functions, for reasons explained below. 
\begin{definition}
    Let $F\colon \R\to\C$ and $x\in\R$. Define
    \[
    T_F(x)
    =
    \sup \left\{\sum_{j=1}^n\left|F\left(x_j\right)-F\left(x_{j-1}\right)\right|: n \in \mathbb{N},-\infty<x_0<\cdots<x_n=x\right\}.
    \]
    $T_F$ is called the total variation function of $F$.
    It is an increasing function with values in $[0, \infty]$. If $T_F(\infty)=\lim _{x \rightarrow \infty} T_F(x)$ is finite, we say that $F$ is of bounded variation on $\mathbb{R}$, and we denote the space of all such $F$ by $BV$ (\cite{follandRealAnalysis}). We denote by $BV_{c}$ the set of functions in $BV$ with compact support.
\end{definition}

\begin{definition}\label{def:definicionDefinitivaDerivadaDebilFrontera}
     Let $\Lambda$ and $\gamma$ be as above. Given $g\in L_{loc}^1(\Lambda)$, we say that $g$ has weak derivative $h=g'$ if for every $\varphi\colon \R\to\R$ in $BV_c(\R)$ we have
    \begin{equation}\label{eq:derivadaDebilCLNstyle}
    \int_{\R}g(t,\gamma(t))\varphi'(t)\, dt
    =
    -\int_{\R}g'(t,\gamma(t))(1+i\gamma'(t))\varphi(t)\, dt.    
    \end{equation}
\end{definition}

\begin{remark}
The reason to choose as test function the class $BV_{c}$ 
is so that $\varphi\circ\Phi_1^{-1}$ is a test function when $\varphi\in C_c^{\infty}(\R)$.
\end{remark}

\begin{lemma}\label{lemma:derivadaDatoDirichlet}
    Let $g\colon \Lambda=\partial\Omega\to \R$. Assume that $g$ has weak derivative $g'$ in the sense of Definition \ref{def:definicionDefinitivaDerivadaDebilFrontera}. Then $(g\circ\Phi)$ has weak derivative $(g\circ\Phi)'$  and is given by
    \begin{equation}\label{eq:expresionDerivada}
    (g\circ\Phi)'
    =
    (g'\circ\Phi)\cdot \Phi'\quad \text{ a.e. in }\R.    
    \end{equation}
\end{lemma}

\begin{proof}
We want to see that if
\begin{equation}\label{eq:derivDebil1}
\int_{\R}g(t,\gamma(t))\varphi'(t)\, dt
=
-\int_{\R}g'(t,\gamma(t))(1+i\gamma'(t))\varphi(t)\, dt,\quad \forall\varphi\in \operatorname{BV}_c(\R),
\end{equation}
then
\begin{equation}\label{eq:derivDebil2}
-\int_{\R}(g\circ\Phi)(t)\varphi'(t)\, dt
=
\int_{\R}(g'\circ\Phi)(t)(\Phi')(t)\varphi(t)\, dt,\quad \forall \varphi\in C_c^{\infty}(\R).    
\end{equation}
Since $\Phi=\eta\circ\Phi_1$, we have
\begin{align*}
    -\int_{\R}(g\circ\Phi)(t)\varphi'(t)\, dt
    = 
    -\int_{\R}(g\circ\eta)(\Phi_1(t))\varphi'(t)\, dt
    =
    -\int_{\R}(g\circ\eta)(\Phi_1(t))\varphi'(\Phi_1^{-1}\circ\Phi_1(t))\, dt,
\end{align*}
Making the change variables $y=\Phi_1(t)$ and applying the chain rule, the above integral becomes
\[
-\int_{\R}(g\circ\Phi)(t)\varphi'(t)\, dt
=
-\int_{\R}(g\circ\eta)(y)(\varphi\circ\Phi_1^{-1})'(y)\, dy.
\]

We would like to apply to this integral the relation \eqref{eq:derivDebil1} with test function $\varphi\circ\Phi_1^{-1}$. We have to check that it is of bounded variation with compact support. The compact support is clear. Besides, since $\Phi_1^{-1}$ is strictly increasing and locally bounded, we have that  $\Phi_1^{-1}$ is of bounded variation in $[\Phi_1(a)-1,\Phi_1(b)+1]$. So, given
$x_1<x_2<\ldots\leq \Phi_1(a)-1<x_k<\ldots\leq \Phi_1(b)+1<x_{k'+1}<\ldots<x_n$, $n\in\N$, we have that 
\[
\sum_{1}^{n}|\varphi(\Phi_1^{-1}(x_j))-\varphi(\Phi_1^{-1}(x_{j-1}))|\le 
\|\varphi'\|_{\infty}(T_{\Phi_1^{-1}}(\Phi_1(b)+1)-T_{\Phi_1^{-1}}(\Phi_1(a)+1))
<
\infty, 
\]
and thus, $\varphi\circ\Phi_1^{-1}$ is of bounded variation in $\R$. Therefore, applying \eqref{eq:derivDebil1} using as test function $\varphi\circ\Phi_1^{-1}$, we obtain
\[
    -\int_{\R}(g\circ\eta)(y)(\varphi\circ\Phi_1^{-1})'(y)\, dy
    =
    \int_{\R}g'(\eta(y))(1+i\gamma'(y))(\varphi\circ\Phi_1^{-1})(y)\, dy
    =
    \int_{\R}g'(\Phi(t))\Phi'(t)\varphi(t)\, dt,
\]
having made the change $y=\Phi_1(t)$ and using that $\Phi_1'(t)(1+i\gamma'(\Phi_1(t)))=\Phi'(t)$ and $\eta\circ\Phi_1=\Phi$. \qedhere
\end{proof}

\subsection{The regularity problem}

Lemma \ref{lemma:derivadaDatoDirichlet}, together with Corollary \ref{corol:relacionNeumannRegularidad}, is what allows us to adapt all the work done for the Neumann problem to the Regularity problem. Indeed, if $g$ is the Dirichlet datum on $\Lambda$, $\mathcal{T}_{D}(g)\coloneqq g\circ\Phi$ is the Dirichlet datum on $\R$. By Lemma \ref{lemma:derivadaDatoDirichlet}, it has weak derivative $\mathcal{T}_{D}(g)'=(g'\circ\Phi)\Phi'$. On the other hand, if $f$ is the Neumann datum on $\Lambda$, then the corresponding Neumann datum on $\R$ is $\mathcal{T}_{N}(g)\coloneqq (g\circ\Phi)|\Phi'|$. Therefore,
\begin{equation}\label{eq:relationNeumannDirichletData}
|\mathcal{T}_{D}(g)'|
=
|\mathcal{T}_{N}(g')|.    
\end{equation}
In particular, their $L^p$ and $L^{p,1}$-norms will be equal. Hence, the same schemes of proof used for the Neumann problem work as well for the Regularity problem, so we just sketch the arguments.

\begin{proof}[Proof of Theorem \ref{thm:regularitySolvability}]
    We begin with the $L^p$-case. The proof is a repetition of Theorem \ref{thm:Lp_Solvability}. Since $f\in \widetilde{L}^1(\Lambda)$, then 
    \[
    u
    \coloneqq 
    u_{f\circ \Phi,D}
    =
    P_y\ast (f\circ\Phi)
    \]
    is a solution to the Dirichlet problem in $\R_{+}^{2}$ with datum $f\circ\Phi$. Besides, by Lemma \ref{lemma:solvabilityRegularityProblemHalfPlane}, since $|\Phi'|^{-p}\Phi(\nu)\in A_p(\R)$ we have 
    \[
    \|\mathcal{M}_{\alpha}(\nabla u_{f\circ\Phi,D})\|_{L^p(\R,|\Phi'|^{-p}\Phi(\nu))}\lesssim \|f\circ\Phi\|_{L^p(\R,|\Phi'|^{-p}\Phi(\nu))}.
    \]
    Then, if we define $v\coloneqq u_{f\circ\Phi,D}\circ\Phi^{-1}$, we have that $v$ is a solution to the Dirichlet problem in $\Omega$ and, reasoning with the norms of the non-tangential maximal operators as in the proof of Theorem \ref{thm:Lp_Solvability}, it satisfies $\|\mathcal{M}_{\alpha}(\nabla v)\|_{L^p(\Lambda,\nu)}\lesssim \|f\|_{L^{p}(\Lambda,\nu)}$. Uniqueness follows as in Theorem \ref{thm:Lp_Solvability}.

    \vspace{0.3cm}

    For the $L^{p,1}$-result, let $\bm{p}\in \{p_{-},p_{+}\}$ be as in the proof of Theorem \ref{thm:Lp1_Solvability}. Let $f\in \widetilde{L}^1(\Lambda)$ with weak derivative $f'\in L^{\bm{p},1}(\Lambda,\nu)$. To solve the $L^{\bm{p},1}(\Lambda,\nu)$-Regularity problem in $\Omega$ with datum $f$, consider in $\R_{+}^{2}$ the Poisson integral 
    \[
    u
    =
    u_{(f\circ\Phi),D}
    =
    P_{y}\ast (f\circ\Phi).
    \]
    This is a solution to the Dirichlet problem in $\R_{+}^{2}$ with datum $(f\circ\Phi)$ that satisfies the estimate:
    \begin{equation}\label{eq:ntEstimateLp1Semiplano}
        \left\|\frac{\mathcal{M}_{\alpha}(\nabla u)}{|\Phi'|}\right\|_{L^{\bm{p},\infty}(\R,\Phi(\nu))}
        \lesssim
        \|(f'\circ\Phi)\|_{L^{\bm{p},1}(\R,\Phi(\nu))}
        =
        \|(f'\circ\Phi)(\Phi')\|_{\mathcal{L}^{\bm{p},1}(\R,(|\Phi'|,\Phi(\nu)))},
    \end{equation}
    where, by Lemma \ref{lemma:calculoGradienteRegularidad}, 
    \begin{equation}\label{eq:gradientLp1InProof}
    \nabla u
    =
    (-P_y\ast (f\circ\Phi)',Q_y\ast (f\circ\Phi)').    
    \end{equation}
    
    The delicate step is to transfer estimate \eqref{eq:ntEstimateLp1Semiplano} to $\Omega$ to obtain
    \begin{equation}\label{eq:estimateLp1RegularityOmega}
    \|\mathcal{M}_{\alpha}(\nabla v)\|_{L^{\bm{p},\infty}(\Lambda,\nu)}
    \lesssim 
    \|f'\|_{L^{\bm{p},1}(\Lambda,\nu)},    
    \end{equation}
    where $v=u\circ\Phi^{-1}$. We have already done this in the proof of Theorem \ref{thm:Lp1_Solvability}. We indicate how to repeat the argument. 
    
    Since $(f\circ\Phi)'=(f'\circ\Phi)(\Phi')\in \mathcal{L}^{\bm{p},1}(\R,(|\Phi'|,\Phi(\nu)))$, then $f'\circ\Phi\in L^{\bm{p},1}(\R,\Phi(\nu))$. Hence, by Lemma 2.2 (b) in \cite{CNO}, there exists a sequence $\{g_m\}_{m\in\N}\subset C_c\cap L^{\bm{p},1}(\R,\Phi(\nu))$ such that 
    \[
    \|g_m-(f'\circ\Phi)\|_{L^{\bm{p},1}(\Lambda,\Phi(\nu))}
    \xrightarrow{m\to\infty}
    0.
    \]
    Then, defining $f_m\coloneqq g_m\cdot (\Phi')\in \mathcal{L}^{\bm{p},1}(\R,(|\Phi'|,\Phi(\nu)))$ for $m\in\N$, we have  
    \[
    \|
    (f\circ\Phi)'-f_m
    \|_{\mathcal{L}^{\bm{p},1}(\R,(|\Phi'|,\Phi(\nu)))}
    =
    \left\|
    \frac{(f'\circ\Phi)\cdot(\Phi')-g_m\cdot(\Phi')}{|\Phi'|}
    \right\|_{L^{\bm{p},1}(\R,\Phi(\nu))}
    \xrightarrow{m\to\infty}
    0.
    \]
    Notice that 
    \[
    F(x,y)
    \coloneqq 
    \partial_1u(x,y)-i\partial_2u(x,y)
    =
    -P_{y}\ast (f\circ\Phi)'(x)
    -iQ_{y}\ast (f\circ\Phi)'(x),\quad (x,y)\in\R_{+}^{2},
    \]
    is holomorphic in $\R_{+}^{2}$ with modulus equal to $|\nabla u|$. So if we define, for each $m\in\N$,
    \[
    F_m(x,y)
    \coloneqq
    -(P_y\ast f_m)(x)-i(Q_y\ast f_m)(x),\quad (x,y)\in\R_{+}^{2},
    \]
    these functions are holomorphic and converge uniformly to $F$ on compact subsets of $\R_{+}^{2}$. Indeed, given $K\subset \R_{+}^{2}$ compact and $z\in K$, we can bound $P_{y}(x-t),Q_{y}(x-t)\lesssim_{K}\frac{1}{1+|t|}$ in the convolution to obtain,
    \[
    |F_m(x,y)-F(x,y)|
    \lesssim_{K}
    \int_{\R}\frac{|f_m(t)-(f\circ\Phi)'(t)|}{1+|t|}\, dt
    \lesssim 
    \left\|\frac{f_m-(f\circ\Phi)'}{|\Phi'|}\right\|_{L^{\bm{p},1}(\Phi(\nu))}
    \left\|\dfrac{\left(\nu\circ\Phi\right)^{-1}}{1+|\cdot|}\right\|_{L^{\bm{p}',\infty}(\Phi(\nu))},
    \]
    which tends to $0$ when $m\to\infty$.
    
    Since we have uniform convergence over compact subsets of $\R_{+}^{2}$, we can repeat the limiting argument in the proof of Theorem 1.5 in \cite{CNO} to obtain \eqref{eq:estimateLp1RegularityOmega}. \qedhere

\end{proof}

\section{Examples for the cone with power weights}
\label{section:examplesDirichlet}

Let $\alpha\in (0,2)$ and consider the graph Lipschitz domain $\Omega$ given by a cone of angle $\alpha\pi$ (Figure \ref{fig:coneAngleAlpha}), with vertex at $0$ and symmetric with respect to the $y$-axis. For this case, the conformal mapping was carefully described in page 15 of \cite{rellich}, obtaining
\[
\Phi(z)
=
e^{i\frac{(1-\alpha)}{2}\pi}z^{\alpha}
=
ie^{-i\frac{\alpha}{2}\pi}e^{\alpha(\log{|z|}+i\operatorname{Arg}(z))},\quad z\in\R_{+}^{2},
\]
where the argument is defined with brach cut $\{iy:y\leq 0\}$, so that $\Phi\colon \R_{+}^{2}\to \Omega$ is holomorphic.

\begin{figure}
    \centering
\begin{tikzpicture}[scale=1]
    \fill[cyan, opacity=0.3] (6.4, 3.2) -- (9.6, 3.2) -- (8, 0.64) -- (6.4, 3.2) -- cycle;
    \fill[cyan, opacity=0.3] (0, 0.64) -- (0, 3.2) -- (3.84, 3.2) -- (3.84, 0.64) -- (0, 0.64) -- cycle;
    \draw[black, ->] (0, 0.64) -- (3.84, 0.64);
    \draw[black, ->] (1.92, 0) -- (1.92, 3.2);
    \draw[black, thick, ->] (4.16, 2.24) .. controls (5.12, 2.88) and (6.08, 2.88) .. (6.72, 2.24);
    \node[anchor=center, font=\normalsize] at (0.4, 2.8) {$\mathbb{R}_{+}^{2}$};
    \node[anchor=center, font=\normalsize] at (8.64, 2.88) {$\Omega$};
    \node[anchor=center, font=\normalsize] at (5.44, 2.94) {$\Phi$};
    \draw[black, ->] (6.4, 0.64) -- (9.6, 0.64);
    \draw[black, ->] (8, 0.32) -- (8, 3.52);
    \draw[black, thick] (6.4, 3.2) -- (8, 0.64) -- (9.6, 3.2);
    \draw[black, thick, ->] (7.676, 1.157) .. controls (7.89, 1.24) and (8.104, 1.237) .. (8.318, 1.149);
    \node[anchor=center, font=\normalsize] at (8, 1.5) {$\alpha\pi$};
\end{tikzpicture}

    \caption{A cone of angle $\alpha\pi$.}
    \label{fig:coneAngleAlpha}
\end{figure}
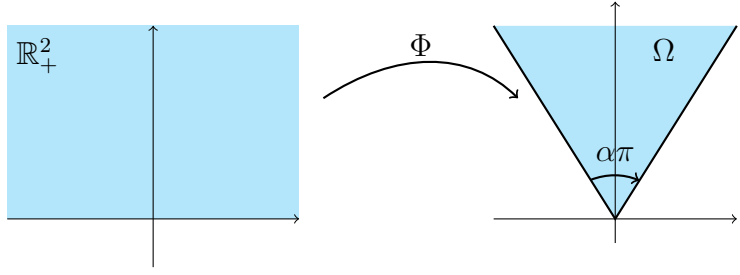

Passing to boundary values, we get
\[
\Phi(x)
=
e^{i\frac{(1-\alpha)}{2}\pi}x^{\alpha},
\quad x\in \R\equiv \partial\R_{+}^{2},
\]
and therefore
\[
|\Phi'(x)|\simeq_{\alpha} |x|^{\alpha-1},\quad x\in\R.
\]

Consider on $\Lambda$ the weight $\nu$  
\begin{equation}\label{eq:examplePowerWeightDirichlet}
\nu(\xi)
\coloneqq 
|\xi|^{\beta},\quad \xi\in\Lambda.    
\end{equation}
Then
\[
\Phi(\nu)(x)
=
\left(\nu\circ\Phi\right)(x)|\Phi'(x)|
=
|x|^{\alpha(\beta+1)-1},\quad x\in\R.
\]

Let us first recall the following result regarding power weights.

\begin{lemma}\label{lemma:powerWeights}
    Let $\beta\in\R$. Then:
    \begin{itemize}
        \item $|x|^{\beta}\in A_{\infty}(\R)$ if and only if $\beta>-1$.
        \item $|x|^{\beta}\in A_{p}(\R)$ if and only if $\beta\in (-1,p-1)$.
        \item $|x|^{\beta}\in A_{p}^{\mathcal{R}}(\R)$ if and only if $\beta\in (-1,p-1]$.
    \end{itemize}
\end{lemma}

Since $\nu\in A_{\infty}(\Lambda)$ if and only if $\Phi(\nu)\in A_{\infty}(\R)$, we need  $\beta>-1$. 

\subsection{\texorpdfstring{$L^{p}$ and $L^{p,1}(\Lambda,\nu)$}{}-solvability for the Dirichlet problem}

\begin{corollary}\label{corol:ejemploDirichlet}
Let  $\alpha\in (0,2)$ and  $\beta>-1$. Then, the Dirichlet problem is solvable in $L^p(\Lambda,\nu)$  with $\nu(\xi)
\coloneqq 
|\xi|^{\beta},$ if and only if $p\in (\alpha(\beta+1),\infty)$ and in $L^{p,1}(\Lambda,\nu)$ for any $p\in [\alpha(\beta+1),\infty)$.  Moreover, for $p=\alpha(\beta+1)$, the Dirichlet problem is not solvable in $L^p(\Lambda,\nu)$. 
\end{corollary}
\begin{proof} We have that   $\Phi(\nu)\in A_{p}$ if and only if $\alpha(\beta+1)-1\in (-1,p-1)$ while $\Phi(\nu)\in A_{p}^{\mathcal{R}}$ if and only if $\alpha(\beta+1)-1\in (-1,p-1]$, that is if $\alpha(\beta+1)\leq p$. The result follows by Theorems \ref{thm:thm4.4_Kenig} and \ref{thm:dirichletEndpointPesos}.
\end{proof}

\subsubsection{Non-solvable Dirichlet-endpoint: cone with power-log measure}

One may wonder if the Dirichlet problem is always solvable in $L^{p,1}(\Lambda,\nu)$ for the endpoint $p=p_{\Phi}(\nu)$. This was already clarified in page 2028 of \cite{CarroOrtiz} for $\nu\equiv 1$, finding a domain $\Omega$ for which the problem is not solvable at $p_{\Phi}(\nu)$. Here we adapt their example, but instead of looking for a pathology in the domain, we fix the domain and choose a measure for which the Dirichlet problem is not solvable at the endpoint.

Consider again the cone of aperture $\alpha$. Define on it the weight $\nu$ 
\begin{equation}\label{eq:measureNonSolvableEndpointDirichlet}
\nu(\xi)
\coloneqq 
\frac{|\xi|^{\beta}}{1+\log_{+}{\frac{1}{|\xi|}}},\quad \xi\in\Lambda.
\end{equation}
Then
\[
 \Phi(\nu)(x)
=
\left(
\nu\circ\Phi
\right)
|\Phi'|
=
\frac{|x|^{\alpha\beta}}{1+\log_{+}\frac{1}{|x|^{\alpha}}}|x|^{\alpha-1}
=
\frac{|x|^{\alpha(\beta+1)-1}}{1+\log_{+}\frac{1}{|x|^{\alpha}}},
\quad x\in\R.
\]
Adapting the computations in page 2028 of \cite{CarroOrtiz}, we have that $p_{\Phi(\nu)}=\alpha(\beta+1)$ with $\Phi(\nu)\not\in A_{p_{\Phi}(\nu)}^{\mathcal{R}}$, and therefore we cannot guarantee $L^{p_{\Phi(\nu)},1}(\Lambda,\nu)$-solvability.

\subsection{\texorpdfstring{$L^{p}(\Lambda,\nu)$}{}-solvability for the Neumann and Regularity problems}

\begin{corollary}\label{corol:exampleLp}
Let $\alpha\in (0,2)$, $\beta>-1$ and $\nu(\xi)\coloneqq |\xi|^{\beta}$ as in Corollary \ref{corol:ejemploDirichlet}. The Neumann and the Regularity problems are solvable in $L^p(\Lambda, \nu)$ in the following cases: 
\begin{itemize}
    \item For $\alpha\in (0,1)$, and for every $p$ such that 
    \[
    \max\{1,\beta+1\}<p<\infty.
    \]
    \item For $\alpha\in (1,2)$ and for every $p$ such that 
    \[
   \max\{1,\beta+1\}< p<     
    \max\left\{1,\frac{\alpha(\beta+1)}{\alpha-1}\right\}.
    \]
\end{itemize}
\end{corollary}

\begin{proof}
We have that
\[
|\Phi'|^{-p}\Phi(\nu)
=
|x|^{\alpha\beta+(\alpha-1)(1-p)}
\in 
A_p(\R)\ \iff\  (\alpha-1)(1-p)+\alpha\beta\in (-1,p-1).
\]

This is equivalent to:
\begin{itemize}
    \item When $\alpha\in (0,1)$, we have $p>\beta+1$.
    \item When $\alpha\in (1,2)$, then $p\in \left(\beta+1,\frac{\alpha}{\alpha-1}(\beta+1)\right)$.
\end{itemize}
Therefore, we have:
\begin{itemize}
    \item For $\alpha\in (0,1)$, 
    \[
    p_{-}=\max\{1,\beta+1\},\quad p_{+}=\infty.
    \]
    \item For $\alpha\in (1,2)$,
    \[
    p_{-}
    =
    \max\{1,\beta+1\}, \quad   p_{+}
    =
    \max\left\{1,\frac{\alpha(\beta+1)}{\alpha-1}\right\},
    \]

\end{itemize}
   and the result follows by Theorem \ref{thm:Lp_Solvability}.
\end{proof}

We notice here  that, in the case $\alpha\in (1,2)$, the range $(p_{+},p_{-})$ can be any subinterval of $(1,\infty)$. Indeed, given any $1\leq a<b<\infty$ and taking 
\[
\begin{cases}
    \beta=a-1\\
    \alpha=\left(\frac{b}{a}-1\right)^{-1}+1,
\end{cases}
\]
the range of solvability is $(a,b)$. Also, it can be the empty set if $a<b<1$.

\subsection{\texorpdfstring{$H_{at}^{1}(\Lambda,\nu)$}{}-solvability for the Neumann problem}\label{section:examplesH1}

\begin{corollary}\label{corol:exampleH1}
Let $\beta>-1$. Then, the Neumann problem is solvable in $H_{at}^{1}(\Lambda, \nu)$ with $\nu(\xi)=|\xi|^{\beta}$  in the following cases: 
\begin{itemize}
    \item For $\alpha\in (0,1)$ and $\beta\in (-1,0]$.
    \item For $\alpha\in (1,2)$ and $\beta\in (-\frac{1}{\alpha},1]$.
\end{itemize}
\end{corollary}

\begin{proof} Clearly $(\nu\circ \Phi)(x)=|x|^{\alpha\beta}$ and the result follows checking when this weight is in $A_1$. 
\end{proof}

\begin{corollary}
    The Neumann problem in a cone is solvable in $H_{at}^1(\Lambda, \nu)$ with 
    \begin{equation*}\label{eq:logWeightExampleH1}
\nu(\xi)
\coloneqq 
\begin{cases}
    \log{\frac{1}{|\xi|}} & \text{ if }|\xi|<\frac{1}{e},\\
    1 & \text{ if }|\xi|\geq \frac{1}{e},
\end{cases}    
\end{equation*}
\end{corollary}

\begin{proof}
$\nu\in A_1(\R)$ as a consequence of Example 7.1.8 in page 507 of \cite{grafakosClassical}. If $w:=\nu\circ\Phi$, we also have to check that $\Phi(\nu)=w(x)|\Phi'(x)|=\max\{1,|\log{\frac{1}{|x|^{\alpha}}}|\}|x|^{\alpha-1}\in A_{\infty}(\R)$. If $\alpha\in [1,2)$, this can be justified briefly. Indeed, since $w\in A_1$, there exists $\gamma>1$ such that $w^{\gamma}\in A_1$. On the other hand, given $\beta>-1$, we have $|x|^{\beta}\in A_p$ for some $p\in (1,\infty)$. Therefore, by interpolation, since $w^{\gamma}\in A_p$ as well, we have
\[
(w^{\gamma}(x))^{\frac{1}{\gamma}}(|x|^{\beta})^{1-\frac{1}{\gamma}}
=
w(x)\cdot |x|^{\beta(1-\frac{1}{\gamma})}
\in A_p.
\]
We want $\alpha-1=\beta(1-\frac{1}{\gamma})$. So $\beta=\frac{\alpha-1}{1-\frac{1}{\gamma}}$. We need $\beta>-1$. If $\alpha\in [1,2)$ this is clear. 

If $\alpha\in (0,1)$, then to prove that $\Phi(\nu)\in A_{\infty}$ we need a more tedious argument using the factorization theorem. We omit it for the sake of brevity.
\end{proof}

\begin{remark}
    This example shows that Theorem \ref{thm:H1solvability} is more general than Theorem 5.1 in
    \cite{LanzaniCapognaBrown}, because it allows us to prove solvability for measures on the boundary that are not of power-type.
\end{remark}

\begin{corollary} Let $w\in A_1(\R)$. Then, the Neumann problem  in a cone with angle $\alpha\pi$ and $\alpha\in [1,2)$  is solvable in $H_{at}^1(\Lambda,\nu)$ with 
\begin{equation}\label{eq:exampleConformalWeight_1}
\nu(\xi)
\coloneqq 
(w\circ \Phi^{-1})(\xi),
\quad
\xi\in\Lambda,
\end{equation}
\end{corollary}
\begin{proof} The requirement of Theorem \ref{thm:H1solvability} is that $w\in A_1(\R)$, and that $(\nu\circ\Phi)|\Phi'|=w\cdot |\Phi'|\in A_{\infty}(\R)$. The first condition is clear and 
also the second condition since $ |\Phi'(x)|^{-1}=|x|^{1-\alpha}\in A_1$, and the result follows by the factorization theorem. 
\end{proof}

\subsection{\texorpdfstring{$L^{p_{\pm},1}(\Lambda,\nu)$}{}-solvability for the Neumann and Regularity problems}\label{sec:examplesLp1}

From Corollary \ref{corol:exampleLp}, we have that 
\[
p_{+}
=
\begin{cases}
    \infty & \text{ if }\alpha\in (0,1],\\
    \max\{1,\frac{\alpha(\beta+1)}{\alpha-1}\} & \text{ if }\alpha\in (1,2).
\end{cases}
\]
That is, $L^{p_{+},1}$-solvability only makes sense if $\alpha\in (1,2)$ and $\alpha \beta>-1$. We also have
\[
p_{-}
=
\max\{1,\beta+1\},
\]
so it makes sense to consider $L^{p_{-},1}$-solvability when $\beta>0$. Under these conditions, we obtain:

\begin{corollary}\label{corol:exampleLp1}
    In the setting of Corollary \ref{corol:exampleLp}, if 
    \[
    \alpha\beta+1>0\quad\mbox{and}\quad \beta>-1,
    \]
the Neumann problem is solvable in     $L^{\frac{\alpha(\beta+1)}{\alpha-1}, 1}(\Lambda,\nu)$   and  in  $L^{\beta+1,1}(\Lambda,\nu)$. 
\end{corollary}

\begin{proof}  We want to use Theorem \ref{thm:sufficientConditionSpR} with $u=|\Phi'|$ and $v=\Phi(\nu)$; that is
\[
u(x)
=
|x|^{\alpha-1},
\quad 
v(x)
=
|x|^{\alpha\beta+\alpha-1},
\quad 
vu^{-1}(x)
=
|x|^{\alpha\beta}.
\]
Hence, we have to check that 
$|x|^{\alpha-1}\in A_{p'}^{\mathcal{R}}(|x|^{\alpha\beta})$ (or  $|x|^{\alpha-1}\in A_{p'}(|x|^{\alpha\beta})$), $|x|^{\alpha\beta}
\in A_p(|x|^{\alpha-1})$ (resp. $|x|^{\alpha\beta}
\in A_p^{\mathcal{R}}(|x|^{\alpha-1})$) and 
\begin{equation}\label{vuv}
\sup_{Q\subset \R}\frac{u(Q)(vu^{-1})(Q)}{|Q|v(Q)}<\infty. 
\end{equation}
Passing to balls using the doubling property and distinguishing cases as in page 506 of \cite{grafakosClassical}, we obtain \eqref{vuv} together with 
\begin{equation*}
|x|^{\gamma}\in A_p(|x|^{\delta})
\iff
\begin{cases}
    \delta>-1\\
    \delta+\gamma>-1\\
    \delta+\gamma-p'\gamma>-1,
\end{cases}
\quad\mbox{ and }\quad
|x|^{\gamma}\in A_p^{\mathcal{R}}(|x|^{\delta})
\iff
\begin{cases}
    \delta>-1\\
    \delta+\gamma>-1\\
    \delta+\gamma-p'\gamma\geq-1,
\end{cases}
\end{equation*}
and thus the result follows. Detailed computations for $|x|^{\gamma}\in A_p(|x|^{\delta})$ can be found in \cite{WickEtAl_Weights}.
\end{proof}

\section{Open problems}

The main open problem is to prove that the imposed conditions  for solvability are optimal:
\begin{enumerate}
    \item \textit{Given $p\in (1,\infty)$ and $\nu\in A_{\infty}(\Lambda)$ such that the Neumann problem is solvable in $X=L^p(\Lambda, \nu)$ with $Y=L^p(\Lambda, \nu)$, is it true that $|\Phi'|^{-p}\Phi(\nu)\in A_p(\R)$?}

    \item \textit{Given $\nu\in A_{\infty}(\Lambda)$ such that the Neumann problem is solvable in $X=H_{at}^{1}(\Lambda, \nu)$ with $Y=L^1(\Lambda, \nu)$, is it true that $\nu\circ\Phi\in A_1(\R)$?}
    
    \item \textit{Given $\nu\in A_{\infty}(\Lambda)$ such that the Neumann problem is solvable in $X=L^{p_{\pm},1}(\Lambda, \nu)$ with $Y=L^{p_{\pm},\infty}(\Lambda, \nu)$, is it true that $(|\Phi'|,\Phi(\nu))\in \mathcal{S}_p^{\mathcal{R}}$?}
\end{enumerate}

The $H_{at}^{1}(\Lambda,\nu)$-solvability of the Regularity problem also remains open.

\section*{Acknowledgements}

We thank the referee for the careful reading and several suggestions that have improved the clarity of the exposition. We also thank Pablo Hidalgo-Palencia for several helpful conversations regarding Partial Differential Equations, and in particular some suggestions that led to Section \ref{sec:duality}.


\begin{thebibliography}{30}

\bibitem{ArmitageHalfSpace}
D.~H. Armitage,
\textit{The Neumann problem for a function harmonic in ${\bf R}^n \times (0,\infty)$},
Arch. Rational Mech. Anal. \textbf{63} (1976), no.~1, 89--105.




\bibitem{AHMMT2020}
J.~Azzam, S.~Hofmann, J.~M. Martell, M.~Mourgoglou, and X.~Tolsa,
\textit{Harmonic measure and quantitative connectivity: geometric characterization of the \(L^p\)-solvability of the Dirichlet problem},
Invent. Math. \textbf{222} (2020), no.~3, 881--993.



\bibitem{AE_FBY}
F.~Ballesta-Yagüe,
\textit{Analytic extensions of $A_{\infty}$-weights on Lipschitz curves and their use in weighted Hardy spaces}, Arxiv Preprint, 2025. Available at: \url{https://arxiv.org/abs/2505.15278}. To appear in \textit{Proc. Amer. Math. Soc.} 



\bibitem{calderonCauchyIntegral}
A.~P. Calderón,
\textit{Cauchy integrals on Lipschitz curves and related operators},
Proc. Nat. Acad. Sci. U.S.A. \textbf{74} (1977), no.~4, 1324--1327.




\bibitem{carroDomingoSalazar}
M.~J. Carro and C.~Domingo-Salazar,
\textit{Stein’s square function \(G_\alpha\) and sparse operators},
J. Geom. Anal. \textbf{27} (2017), no.~2, 1624--1635.




\bibitem{zarembaCarro}
M.~J. Carro, T.~Luque, and V.~Naibo,
\textit{The Zaremba problem in two-dimensional Lipschitz graph domain}, Trans. Amer. Math. Soc. \textbf{378} (2025), no.~10, 6885--6911.





\bibitem{CNO}
M.~J. Carro, V.~Naibo, and C.~Ortiz-Caraballo,
\textit{The Neumann problem in graph Lipschitz domains in the plane},
Math. Ann. \textbf{385} (2023), no.~1--2, 17--57.



\bibitem{rellich}
M.~J. Carro, V.~Naibo, and M.~Soria-Carro,
\textit{Rellich identities for the Hilbert transform},
J. Funct. Anal. \textbf{286} (2024), no.~4, Paper No. 110271, 22 pp.



\bibitem{transmissionCarro}
M.~J. Carro, V.~Naibo, and M.~Soria-Carro,
\textit{Transmission problems for simply connected domains in the complex plane}, J. Differential Equations \textbf{433} (2025), Paper No.~113216, 35~pp.





\bibitem{CarroOrtiz}
M.~J. Carro and C.~Ortiz-Caraballo,
\textit{On the Dirichlet problem on Lorentz and Orlicz spaces with applications to Schwarz–Christoffel domains},
J. Differential Equations \textbf{265} (2018), no.~5, 2013--2033.





\bibitem{ChungHuntKurtz}
H.~M. Chung, R.~A. Hunt, and D.~S. Kurtz,
\textit{The Hardy-Littlewood maximal function on {$L(p,\,q)$} spaces with weights},
Indiana Univ. Math. J. \textbf{31} (1982), no.~1, 109--120.




\bibitem{coifmanWeiss}
R.~R. Coifman and G.~Weiss,
\textit{Extensions of Hardy spaces and their use in analysis},
Bull. Amer. Math. Soc. \textbf{83} (1977), no.~4, 569--645.



\bibitem{SawyerCUMP}
D.~Cruz-Uribe, J.~M. Martell, and C.~Pérez,
\textit{Weighted weak-type inequalities and a conjecture of Sawyer},
Int. Math. Res. Not. \textbf{2005} (2005), no.~30, 1849--1871.




\bibitem{CruzUribeSweeting}
D.~Cruz-Uribe and B.~Sweeting,
\textit{Weighted weak-type inequalities for maximal operators and singular integrals},
Rev. Mat. Complut. \textbf{38} (2025), no.~1, 183--205.






\bibitem{DahlbergHarmonicMeasure}
B.~E. J. Dahlberg,
\textit{Estimates of harmonic measure},
Arch. Rational Mech. Anal. \textbf{65} (1977), no.~3, 275--288.



\bibitem{DahlbergKenig}
B.~E.~J. Dahlberg and C.~E. Kenig,
\textit{Hardy spaces and the Neumann problem in {$L^p$} for Laplace's equation in Lipschitz domains},
Ann. of Math. (2) \textbf{125} (1987), no.~3, 437--465.





\bibitem{FJK1984}
E.~B. Fabes, D.~S. Jerison, and C.~E. Kenig,
\textit{Necessary and sufficient conditions for absolute continuity of elliptic-harmonic measure},
Ann. of Math. (2) \textbf{119} (1984), no.~1, 121--141.







\bibitem{follandRealAnalysis}
G.~B. Folland,
\textit{Real analysis}, 2nd ed.,
Pure and Applied Mathematics (New York), John Wiley \& Sons, Inc., New York, 1999.




\bibitem{GarciaCuerva}
J.~García-Cuerva,
\textit{Weighted \(H^p\) spaces},
Dissertationes Math. \textbf{162} (1979), 63 pp.



\bibitem{garnett}
J.~B. Garnett,
\textit{Bounded Analytic Functions},
Graduate Texts in Mathematics, vol. 236, Springer, New York, 2007.



\bibitem{grafakosClassical}
L.~Grafakos,
\textit{Classical Fourier Analysis}, 3rd ed.,
Graduate Texts in Mathematics, vol. 249, Springer, New York, 2014.






\bibitem{HMW}
R.~Hunt, B.~Muckenhoupt, and R.~Wheeden,
\textit{Weighted norm inequalities for the conjugate function and Hilbert transform},
Trans. Amer. Math. Soc. \textbf{176} (1973), 227--251.




\bibitem{JK1982}
D.~S. Jerison and C.~E. Kenig,
\textit{Boundary behavior of harmonic functions in nontangentially accessible domains},
Adv. Math. \textbf{46} (1982), no.~1, 80--147.


\bibitem{jerisonKenig}
D.~S. Jerison and C.~E. Kenig,
\textit{Hardy spaces, \(A_\infty\), and singular integrals on chord-arc domains},
Math. Scand. \textbf{50} (1982), no.~2, 221--247.




\bibitem{kenigWeighted}
C.~E. Kenig,
\textit{Weighted \(H^p\) spaces on Lipschitz domains},
Amer. J. Math. \textbf{102} (1980), no.~1, 129--163.



\bibitem{KermanTorchinsky}
R.~A. Kerman and A.~Torchinsky,
\textit{Integral inequalities with weights for the Hardy maximal function},
Studia Math. \textbf{71} (1981/82), no.~3, 277--284.



\bibitem{LanzaniCapognaBrown}
L.~Lanzani, L.~Capogna, and R.~M. Brown,
\textit{The mixed problem in {$L^p$} for some two-dimensional Lipschitz domains},
Math. Ann. \textbf{342} (2008), no.~1, 91--124.


\bibitem{LernerSparse}
A.~K. Lerner,
\textit{On pointwise estimates involving sparse operators},
New York J. Math. \textbf{22} (2016), 341--349.

\bibitem{LernerNazarov}
A.~K. Lerner and F.~Nazarov,
\textit{Intuitive dyadic calculus: the basics},
Expo. Math. \textbf{37} (2019), no.~3, 225--265.


\bibitem{WickEtAl_Weights}
J.~Li, C.-W. Liang, C.-Y. Shen, and B.~D. Wick,
\textit{Muckenhoupt-type weights and quantitative weighted estimates in the Bessel setting}, Math. Z. \textbf{309} (2025), no.~1, Paper No.~13, 29~pp.





\bibitem{LiOmbrosiPerez}
K.~Li, S.~Ombrosi, and C.~Pérez,
\textit{Proof of an extension of E. Sawyer's conjecture about weighted mixed weak-type estimates},
Math. Ann. \textbf{374} (2019), no.~1--2, 907--929.


\bibitem{MTregularity}
M.~Mourgoglou and X.~Tolsa,
\textit{The regularity problem for the Laplace equation in rough domains},
Duke Math. J. \textbf{173} (2024), no.~9, 1731--1837.

\bibitem{MTneumann}
M.~Mourgoglou and X.~Tolsa,
\textit{Solvability of the Neumann problem for elliptic equations in chord-arc domains with very big pieces of good superdomains},
Preprint (2025), \url{https://arxiv.org/abs/2407.20385}.


\bibitem{M1972}
B.~Muckenhoupt,
\textit{Weighted norm inequalities for the Hardy maximal function},
Trans. Amer. Math. Soc. \textbf{165} (1972), 207--226.



\bibitem{PereyraSparse}
M.~C. Pereyra,
\textit{Dyadic harmonic analysis and weighted inequalities: the sparse revolution},
in \textit{New trends in applied harmonic analysis. Vol. 2—harmonic analysis, geometric measure theory, and applications},
Appl. Numer. Harmon. Anal., Birkhäuser/Springer, Cham, 2019, pp.~159--239.





\bibitem{perezRoure}
C.~Pérez and E.~Roure-Perdices,
\textit{Sawyer-type inequalities for Lorentz spaces},
Math. Ann. \textbf{383} (2022), no.~1--2, 493--528.




\bibitem{steinSingularIntegrals}
E.~M. Stein,
\textit{Singular integrals and differentiability properties of functions},
Princeton Mathematical Series, no. 30, Princeton Univ. Press, Princeton, NJ, 1970.



\bibitem{steinHarmonic}
E.~M. Stein,
\textit{Harmonic Analysis: Real-Variable Methods, Orthogonality, and Oscillatory Integrals},
Princeton Mathematical Series, vol. 43, Princeton Univ. Press, Princeton, NJ, 1993.


\bibitem{VerchotaLayer}
G.~Verchota,
\textit{Layer potentials and regularity for the Dirichlet problem for Laplace's equation in Lipschitz domains},
J. Funct. Anal. \textbf{59} (1984), no.~3, 572--611.


\end{thebibliography}
\end{document}